\documentclass[final]{siamltex}

\usepackage{graphicx}
\usepackage{amsmath}
\usepackage{amssymb}

\title{Stability theory for difference approximations \\of Euler Korteweg equations \\and application to thin film flows\thanks{Research of P.N. is partially supported by the French ANR Project no. ANR- 09-JCJC-0103-01.}}

\author{Pascal NOBLE\thanks{ Institut de Math\'ematiques de Toulouse, UMR CNRS 5219, INSA de Toulouse, 135 avenue de Rangueil, 31077 Toulouse Cedex 4, France  ({\tt pascal.noble@math.univ-toulouse.fr}).}
        \and Jean-Paul VILA\thanks{  Institut de Math\'ematiques de Toulouse, UMR CNRS 5219, INSA de Toulouse, 135 avenue de Rangueil, 31077 Toulouse Cedex 4, France   ({\tt vila@insa-toulouse.fr}).}}

\begin{document}

\maketitle

\begin{abstract}
 We study the stability of various difference
approximations of the Euler Korteweg equations.  This system of evolution PDEs 
is a classical isentropic Euler system perturbed by a dispersive (third order) term. The Euler
equations are discretized with a classical scheme (e.g. Roe, Rusanov
or Lax Friedrichs scheme) whereas the dispersive term is discretized
with centered finite differences. We first prove that a certain amount
of numerical viscosity is needed for a difference scheme to be stable
in the Von Neumann sense. Then we consider the entropy stability of
difference approximations. For that purpose, we introduce an additional unknown,
the gradient of a function of the density. The Euler Korteweg system is transformed
into a hyperbolic system perturbed by a second order skew symmetric
term. We prove entropy stability of Lax Friedrichs type schemes under
a suitable Courant-Friedrichs-Levy condition.  In addition, we propose a spatial 
discretization of the Euler Korteweg system seen as a Hamiltonian system of evolution PDEs. 
This spatial discretization  preserves the Hamiltonian structure and thus is naturally entropy conservative.
We validate our approach numerically on a shallow water system  with surface tension which models thin films.
\end{abstract}



\pagestyle{myheadings}
\thispagestyle{plain}
\markboth{ENTROPY STABLE SCHEMES FOR CAPILLARY FLUIDS}{}%{\sinum}

\section{Introduction}

This paper is motivated by the numerical simulation of the so-called
Euler Korteweg system, which arises in the modeling of capillary
fluids: these comprise liquid-vapor mixtures (for instance highly
pressurized and hot water in nuclear reactors cooling system) \cite{JTB}, superfluids (Helium near
absolute zero) \cite{HAC}, or even regular fluids at sufficiently small scales
(think of ripples on shallow water or other thin films) \cite{LG}. In one space
dimension, the most general form of the Euler Korteweg system we consider
is 
\begin{equation}
\begin{array}{ll}
{\displaystyle \partial_{t}\rho+\partial_{x}(\rho u)=0,}\\
{\displaystyle \partial_{t}(\rho u)+\partial_{x}\left(\rho u^{2}+P(\rho)\right)=\partial_{x}\left(\rho\kappa(\rho)\partial_{xx}\rho+(\rho\kappa'(\rho)-\kappa(\rho))\frac{(\partial_{x}\rho)^{2}}{2}\right),}
\end{array}\label{E-Kg}
\end{equation}
where $\rho$ denotes the fluid density, $u$ the fluid velocity$,P(\rho)$
the fluid pressure and $\kappa(\rho)>0$ the capillary coefficient.
We assume  that $P'(\rho)>0$ for all $\rho>0$ so that the Euler system is always hyperbolic.
In quantum hydrodynamics, the capillary coefficient is chosen so that
$\rho\kappa(\rho)=constant$ \cite{CDS} whereas for classical applications, like
thin film flows, it is often chosen to be constant \cite{BDL}.
% In particular,
%we will focus, for numerical computations, to the case of thin film
%flows down an incline modeled by the shallow water equations with
%surface tension and source term 
%\begin{equation}%\label{SW}
%\begin{array}{ll}
%{\displaystyle \partial_{t}h+\partial_{x}(hu)=0,}\\
%{\displaystyle \partial_{t}(hu)+\partial_{x}\left(hu^{2}+P(h,A_1)\right)=A_1\left(gh\sin(\theta)-3\frac{\nu u}{h}+\frac{\sigma}{\rho}h\partial_{xxx}h\right),}\label{SW}
%\end{array}
%\end{equation}
%\begin{equation}
%\begin{array}{ll}
%{\displaystyle \partial_{t}h+\partial_{x}(hu)=0,}\\
%{\displaystyle {\color{red}\partial_{t}(hu)+\partial_{x}\left(hu^{2}+\frac{5g{\rm c}\, h^{2}}{9}\right)=\frac{10}{9}\left(gh{\rm s}-3\frac{\nu u}{h}+\frac{\sigma}{\rho}h\partial_{xxx}h\right).}}
%\end{array}\label{SW}
%\end{equation}
%\noindent where $h$ is the fluid height, $u$ the streamwise velocity
%averaged along the cross stream direction. The constant $\rho,\nu,\sigma$
%are respectively the density, viscosity and capillarity of the fluid
%used in the Liu-Gollub experiment \cite{LG} whereas $g$ is the constant
%of gravity and $\theta$ is the slope of the channel. The constant $A_1>0$ is arbitrary and $P(h,A_1)$ is a smooth pressure term. 
The Euler Korteweg system (\ref{E-Kg}) falls in the class
of abstract Hamiltonian systems of evolutions PDEs when it is written
with variables $\rho,u$: 
\begin{equation}
{\displaystyle \partial_{t}U=\mathcal{J}\left({\rm E}\mathcal{H}[U]\right),}\label{E-Kh}
\end{equation}
with $U=(\rho,u)^{T}$, $\mathcal{J}=\partial_{x}{\rm J}$, 
\[
{\displaystyle {\rm J}=\left(\begin{array}{cc}
0 & -1\\
-1 & 0
\end{array}\right),\quad\mathcal{H}[U]=\frac{\rho u^{2}}{2}+F(\rho)+\kappa(\rho)\frac{(\partial_{x}\rho)^{2}}{2}=\frac{\rho u^{2}}{2}+\mathcal{E}(\rho,\partial_{x}\rho),}
\]
and ${\rm E}$ denotes the Euler operator 
\[
{\displaystyle {\rm E}\mathcal{H}[U]=
%\left(\begin{array}{c}
%{\displaystyle \frac{u^{2}}{2}+F'(\rho)+\kappa'(\rho)\frac{(\partial_{x}\rho)^{2}}{2}-\partial_{x}(\kappa(\rho)\partial_{x}\rho)}\\
%\rho u
%\end{array}\right)=
\left(\begin{array}{c}
{\displaystyle \frac{u^{2}}{2}+{\rm E}_{\rho}\mathcal{E}(\rho,\partial_{x}\rho)}\\
\rho u
\end{array}\right),\: {\rm E}_{\rho}\mathcal{E}(\rho,\partial_{x}\rho)=F'(\rho)+\kappa'(\rho)\frac{(\partial_{x}\rho)^{2}}{2}-\partial_{x}(\kappa(\rho)\partial_{x}\rho).}
\]
The pressure $P$ is related to $F$ through the relation $\rho F'(\rho)-F(\rho)=P(\rho)$.
Due to the invariance of the equations with respect to spatial 
and time translations, the system (\ref{E-Kh}) admits, via Noether's
theorem, two additional conservation laws which are nothing but the
conservation of momentum (the second equation of (\ref{E-Kg}))
and the conservation of energy: 
%\[
%{\displaystyle \partial_{t}(\rho u)+\partial_{x}\left(\rho u^{2}+P(\rho)\right)=\partial_{x}\left(\rho\kappa(\rho)\partial_{xx}\rho+(\rho\kappa'(\rho)-\kappa(\rho))\frac{(\partial_{x}\rho)^{2}}{2}\right),}
%\]
\begin{equation}
{\displaystyle \partial_{t}\left(\frac{1}{2}\rho u^{2}+\mathcal{E}(\rho,\partial_{x}\rho)\right)+\partial_{x}\left(\frac{1}{2}\rho u^{3}+\rho u{\rm E}_{\rho}\mathcal{E}(\rho,\partial_{x}\rho)+\kappa(\rho)\partial_{x}(\rho u)\partial_{x}\rho\right)=0.}\label{EK-en}
\end{equation}
As a consequence, if the system (\ref{E-Kg}) is set on the real
line or with periodic boundary conditions, the ``entropy'' $\mathcal{H}$ is conserved. Therefore, it is desirable from a numerical
point of view that a difference approximation of (\ref{E-Kg}) or
(\ref{E-Kh}) preserves the energy or, at least, dissipates energy.
In the first case, the difference approximation is an ``entropy conservative''
scheme and in the later case, it is an ``entropy stable'' scheme.

There are two possible strategies to tackle
this problem. The first one consists in considering (\ref{E-Kg})
as a dispersive perturbation of the classical isentropic Euler equations.
This is the point of view adopted e.g. in \cite{LMR}. Here the authors
construct fully discrete entropy conservative scheme for systems of
conservation laws (hyperbolic or hyperbolic-elliptic) endowed with
an entropy-entropy flux pair. These difference approximations are
second and third order accurate and can in turn be used to construct
a numerical method for the computation of weak solutions containing
non- classical regularization-sensitive shock waves. In particular,
the authors considered dissipative/dispersive regularizations that
are linear in the entropy variables 
\[
{\displaystyle \partial_{t}u(v^{\varepsilon})+\partial_{x}f(u(v^{\varepsilon}))=\varepsilon B_{2}\partial_{xx}v^{\varepsilon}+\varepsilon^{2}B_{3}\partial_{xxx}v^{\varepsilon},\quad0<\varepsilon\ll1}
\]
with $B_{i}$ constant symmetric matrices, $B_{2}$ being positive definite. Thus the dispersive terms do not contribute
in the energy equation: 
\[
\displaystyle \partial_{t}\int_{\mathcal{D}} U(u^{\varepsilon})\leq 0
\]
where $\mathcal{D}=\mathbb{R}$ or $\mathcal{D}=\mathbb{R}/L\mathbb{Z}$ ($L>0$) and $U$ is the entropy associated to the system of conservation
laws 
\[
{\displaystyle \partial_{t}u+\partial_{x}f(u)=0.}
\]
This situation contrasts with the one met in the Euler Korteweg system
where the dispersive terms have a contribution in the energy balance
\[
{\displaystyle \partial_{t}\int_{\mathcal{D}}\frac{1}{2}\rho u^{2}+F(\rho)+\kappa(\rho)\frac{(\partial_{x}\rho)^{2}}{2}\,dx=0.}
\]
As a consequence, an entropy conservative or entropy stable scheme
for the isentropic Euler equations coupled with a centered approximation
of dispersive terms may not provide an entropy conservative nor entropy
stable scheme for the Euler Korteweg system. This issue was considered in \cite{CL} where Euler Korteweg equations are written in lagrangian coordinates of mass: by introducing an extended formulation of the system, the authors derived a family of high order and entropy conservative semi-discrete schemes. With these high order approximations schemes in hand, the authors then computed kinetic relations for Van der Waals fluids. In \cite{HR}, an alternative reduction of order of the Euler Korteweg system in lagrangian coordinates of mass is introduced to derive a semi-discrete entropy conservative scheme based on local Galerkin discontinuous methods. Though, the lagrangian coordinates of mass can not be used in dimension $d$ with $d\geq 2$ and one has to consider an alternative extended formulation in Eulerian coordinates: this latter point of view will be expanded here, based on the extended formulation found in \cite{BDD,BDDd}.

In section \ref{sec2}, we consider the stability of various
difference approximations of Euler Korteweg equations in the Von Neumann
sense. We shall prove that even at that linear level, the Godunov
scheme (explicit and implicit in time) is always unstable. In this
direction, we checked the stability of Lax Friedrichs type schemes:
we show that it is stable in the Von Neumann sense under a suitable
Courant-Friedrichs-Levy (CFL) condition for explicit forward Euler
(resp. Runge Kutta) time discretization for first order (resp. second
order) difference schemes. This analysis provides necessary conditions
of stability for the simulations of the fully nonlinear system. Finally
we show that the backward Euler and Crank Nicolson time discretization 
is always stable for Lax Friedrichs type schemes.

 In section \ref{sec3}, we move to the entropy (nonlinear)
stability problem. It is a hard problem to obtain directly entropy
stability from nonlinear difference approximation of Euler Korteweg
equations since discrete integration by parts and time discretization
do not commute. Here, we introduce an additional variable $w=\sqrt{\kappa(\rho)}\partial_{x}\rho/\sqrt{\rho}$
and derive a conservation law for $w$. In this new formulation, the
capillary term appears as an anti dissipative term in the system for
$(u,w)$ and one can prove the well posedness of the Euler Korteweg
system \cite{BDD}. Moreover, the derivation of the energy estimate
follows the same line as a classical energy estimate in the isentropic
Euler equations. In that setting, we show that difference approximations
made of a Lax Friedrichs type (entropy stable) scheme for the hyperbolic
part and centered difference for the anti-diffusive part are entropy
stable under a suitable CFL condition for explicit forward Euler time
discretization and always stable for implicit backward Euler time
discretization. We also introduce an alternative method to obtain directly
entropy conservative scheme. For that purpose, we write the Euler Korteweg system
as a Hamiltonian system of PDEs: by discretizing directly the Hamiltonian, we obtain
a semi discrete scheme that is also Hamiltonian and entropy is trivially preserved.
 Then, one is left with the problem of time discretization: the explicit forward
Euler is always unstable and one has to consider implicit time discretization
to obtain an entropy stable scheme

Finally, in section \ref{sec4}, we carry out numerical
simulations of shallow water equations with surface tension which is a particular case of the 
Euler Korteweg equations. We first
consider thin film flow over a flat bottom and neglect source terms
so as to compare entropy stability of difference approximations for
shallow water in original form and for its new formulation counterpart.
{\it The numerical simulations clearly show that the discretization of the
extended formulation of shallow water equations has better entropy stability
properties}. Then, we consider the difference approximation of the shallow water
equations written as a Hamiltonian system of evolution PDEs. The numerical 
simulation of this Hamiltonian system shows that the dynamical behavior is completely
changed in comparison to entropy stable schemes. Indeed, this Hamiltonian
difference approximation has no numerical viscosity, so that one can
observe the formation of so called ``dispersive shock waves'' \cite{E,EGK,EGS}.
Here, the classic hyperbolic shocks are regularized by dispersive
effects and an oscillatory zone appear and grows with time. We conclude
this section with numerical simulation of a Liu Gollub experiment
\cite{LG} modeled by a consistent shallow water model with source
term derived in \cite{NV}. The numerical simulation show very good agreement with
the experiments in \cite{LG}.

%%%%%%%%%%%%%%%%%%%%%%SECTION 2%%%%%%%%%%%%%%%%%
\section{\label{sec2} Von Neumann stability of difference schemes}

\mathversion{normal}

In this section, we study the Von Neumann stability of various difference approximations.
We  consider two classes of spatial discretizations, namely Godunov and Lax Friedrichs 
schemes for the first order part of the equations whereas the dispersive term is discretized
 with classical centered difference approximations. We also consider second order accurate 
 schemes, namely MUSCL scheme with a Lax Friedrichs type flux for spatial discretization
together with  Runge Kutta (second order accurate) or Crank Nicolson time discretization.

\subsection{Stability of first order accurate schemes}

In this section, we prove that Godunov space discretization are always unstable whereas Lax Friedrichs type scheme are stable under CFL conditions.

\subsubsection{Formulation of the stability problem}

In order to study the Von Neumann stability, we first linearize the Euler Korteweg equations about a constant state $(\rho, q)=(\bar\rho,\bar\rho\bar u)$:
\begin{equation}\label{ek-lin}
\displaystyle
\partial_t v+A\partial_x v=B\partial_{xxx}v,\quad A=\left(\begin{array}{cc} 0 & 1\\ \bar c^2-\bar u^2 & 2\bar u\end{array}\right),\quad B=\left(\begin{array}{cc} 0 & 0\\ \bar\sigma & 0\end{array}\right)
\end{equation}
\noindent
with $v=(\rho,q)^T$, $\bar\sigma=\bar\rho\kappa(\bar\rho)$ and $\bar c=\sqrt{P'(\bar\rho)}$.  We discretize space and denote $v_j(t)$ the approximate value of $v(t,j\delta x), j\in\mathbb{Z}$ and $\delta x>0$. We also introduce the Fourier transform of a sequence $v\in l^2(\mathbb{Z})$: 
$$
\displaystyle
\hat{v}(\theta)=\sum_{k\in\mathbb{Z}}e^{ik\theta}v_k,\qquad \|v\|_{l^2(\mathbb{Z})}=\|\hat v\|_{L^2_{per}(0, 2\pi)}.
$$
\noindent
In what follows, we consider Godunov, Lax Friedrichs and Rusanov discretization of the first order part whereas the capillary term is discretized with centered difference. These schemes have the common formulation
{\setlength\arraycolsep{1pt}
\begin{eqnarray}
\displaystyle
\frac{dv_j}{dt}+\frac{A}{2\delta x}\left(v_{j+1}-v_{j-1}\right)&=&\frac{Q}{2\delta t}\left(v_{j+1}-2v_j+v_{j-1}\right)\nonumber\\
\label{ek-lind}
\displaystyle
&&+\frac{B}{2\delta x^3}\left(v_{j+2}-2v_{j+1}+2v_{j-1}-v_{j-2}\right),
\end{eqnarray}}

\noindent
with $Q={\rm Id}$ for Lax Friedrichs scheme, $Q=(1/2){\rm Id}$ for the modified Lax Friedrichs scheme, $Q=(\delta t/\delta x)\rho(A)$ for Rusanov ($\rho(A)=\max\{|\lambda|,\lambda\in{\rm Sp}(A)\}$), $Q=(\delta t/\delta x)|A|$ for Godunov scheme with 
$$
\displaystyle
|A|=\frac{1}{2\bar c}\left(\begin{array}{cc}\displaystyle |\bar u-\bar c|(\bar u+\bar c)-|\bar u+\bar c|(\bar u-\bar c) & |\bar u+\bar c|-|\bar u-\bar c|,\\
\displaystyle (\bar c^2-\bar u^2)\left(|\bar u+\bar c|-|\bar u-\bar c|\right) & |\bar u+\bar c|(\bar u+\bar c)-|\bar u-\bar c|(\bar u-\bar c)\end{array}\right)
$$
\noindent
We apply the Fourier Transform to (\ref{ek-lind}): $\hat v$ satisfies the differential system:
$$
\displaystyle
\frac{d\hat v}{dt}=i\frac{2\sin(\theta/2)\cos(\theta/2)}{\delta x}A\hat v-\frac{\delta x^2}{2\delta t}\left(\frac{2\sin(\theta/2)}{\delta x}\right)^2Q\hat v+i\cos(\theta/2)\left(\frac{2\sin(\theta/2)}{\delta x}\right)^3B\hat v.
$$ 
\noindent
We introduce the Fourier variable $\xi=2\sin(\theta/2)/\delta x$ so that $\hat v$ satisfies 
\begin{equation}\label{ek-sp-disc}
\displaystyle
\frac{d\hat v}{dt}=\left(\sqrt{1-\frac{(\xi\delta x)^2}{4}}\left(i\xi A+i\xi^3 B\right)-\xi^2\left(\frac{\delta x^2}{2\delta t}\right)Q\right)\hat v:=i\xi M(\xi,\delta x)\hat v,
\end{equation}
\noindent
with matrix $M(\xi,\delta x)$ defined as 
$$
\displaystyle
M(\xi,\delta x)=\zeta\left(A+\xi^2 B\right)+i\xi\frac{\delta x^2}{2\delta t}Q,\quad \zeta=\sqrt{1-\frac{(\xi\delta x)^2}{4}}.
$$
\noindent
In what follows, we consider the stability of the forward Euler, backward Euler and $\Theta$ scheme time discretization of (\ref{ek-sp-disc}): it  reads, respectively, for all $n\geq 0$
$$
\begin{array}{lll}
\displaystyle
\hat v^{n+1}=\left({\rm Id} +i\xi\delta t M(\xi,\delta x)\right)\hat{v}^n, \qquad (FE),\\
\displaystyle
\hat v^{n+1}=\left({\rm Id} -i\xi\delta t M(\xi,\delta x)\right)^{-1}\hat{v}^n, \qquad (BE),\\
\displaystyle
\hat v^{n+1}=\left({\rm Id} -i\Theta \xi\delta t M(\xi,\delta x)\right)^{-1}\left({\rm Id} +i(1-\Theta) \xi\delta t M(\xi,\delta x)\right)\hat{v}^n, \qquad (\Theta\,S).
\end{array}
$$
\noindent
We denote $\Lambda_{\pm}(\xi,\delta x)=R_{\pm}(\xi,\delta x)+i I_{\pm}(\xi,\delta x)$ the eigenvalues of $M(\xi,\delta x)$. The proof of the following proposition is straightforward and left to the reader:
\begin{proposition}\label{stab-1st-o}
A necessary condition for the (FE), (BE) and $(\Theta S)$ time discretizations to be stable is
\begin{equation}\label{cond1}
\displaystyle
\xi I_{\pm}(\xi,\delta x)\geq 0,\quad \forall \xi\in\left[-\frac{2}{\delta x},\,\frac{2}{\delta x}\right].
\end{equation}
\noindent
This condition is sufficient for (BE) scheme and $(\Theta S)$ scheme for all $\Theta\geq 1/2$.  The $(\Theta S)$ and (FE)  scheme (which corresponds to the $0 S$ scheme) are stable under the condition:
\begin{equation}\label{CFL1}
\displaystyle
\delta t\leq \frac{2\xi\, I_{\pm}(\xi,\delta x)}{(1-2\Theta)\xi^2(R_{\pm}^2(\xi,\delta x)+I_{\pm}^2(\xi,\delta x))},\quad \forall \xi\in\left[-\frac{2}{\delta x},\,\frac{2}{\delta x}\right].
\end{equation}
\end{proposition}

\noindent
{\sc Remark.} Note that for $\delta x\to 0$, the condition $\xi I_{\pm}(\xi, 0)\geq 0$ for all $\xi\in\mathbb{R}$ is nothing but the dissipativity of the operator
$Mv:=-A\partial_x v+(\delta x^2\,Q/(2\delta t))\partial_{xx}v+B\partial_{xxx}v$.

\subsubsection{Stability/Instability of first order schemes}

We are now in a position to prove the instability of Godunov/Roe type scheme. In this section, we will have to consider various Courant-Friedrichs-Lewy conditions (denoted CFL condition): we introduce $\lambda_j=\delta t/\delta x^j$ for $j=1,2,3$.\\

\begin{proposition}
Assume $Q=\lambda_1|A|$ (Roe/Godunov scheme), then for $\xi\delta x>0$ fixed and as $|\xi|\to\infty$, one has
$$
\displaystyle
\xi I_{\pm}(\xi,\delta x)\sim \pm\xi|\xi|\left(\sqrt{\bar\sigma^2\zeta^4+2\bar\sigma\zeta\sqrt{1-\zeta^2}|A|_{12}}-\bar\sigma\zeta^2\right).
$$ 
\noindent
As a consequence, the Godunov/Roe space discretization is always unstable regardless to the (FE), (BE) and $(\Theta S)$ time discretizations.
\end{proposition}\\

\noindent
{\sc Remark} The previous proposition also proves that the PDE
$$
\displaystyle
\partial_t v+A\partial_x v=\frac{\delta x}{2}|A|\partial_{xx} v+B\partial_{xxx}v
$$
is {\it ill-posed} in $L^2(\mathbb{R})$, it is therefore hopeless to find a stable scheme for Godunov/Roe spatial discretizations. This is the main difference
between the scalar case where the numerical viscosity induces dissipation and the system case where numerical viscosity interacts with surface tension and 
leads to instability/ill-posedness.\\

\noindent
\begin{proof}
Set $\displaystyle \zeta=\sqrt{1-\frac{\xi^2\delta x^2}{4}}$ and $q_{ij}=\delta x|A|_{ij}/2$, then the eigenvalues $\Lambda_{\pm}(\delta x,\xi)$ of $M(\xi,\delta x)$ are written as
{\setlength\arraycolsep{1pt}
\begin{eqnarray}
\displaystyle
\Lambda_{\pm}(\xi,\delta x)&=&\zeta\bar u+i\xi\frac{q_{11}+q_{22}}{2}\pm\sqrt{\Delta},\nonumber\\
\displaystyle
\Delta&=&\zeta^2(\bar c^2+\bar\sigma\xi^2)-\xi^2\left(\frac{(q_{11}-q_{22})^2}{4}+q_{12}q_{21}\right)\nonumber\\
\displaystyle
&&+i\xi\zeta\left(\bar u(q_{22}-q_{11})+q_{21}+q_{12}(\bar c^2-\bar u^2)+\bar \sigma q_{12}\xi^2\right).\nonumber
\end{eqnarray}
}
\noindent
As $\delta x\to 0$, $\Lambda_{\pm}(\xi,\delta x)$ expands as
$$
\displaystyle
\Lambda_{\pm}(\xi,\delta x)=\pm\sqrt{\bar\sigma\zeta}\sqrt{\zeta+i\sqrt{1-\zeta^2}|A|_{12}}|\xi|+O(1).
$$
\noindent
Then, one finds
$$
\displaystyle
\xi I_{\pm}(\xi,\delta x)=\pm\xi|\xi|\left(\sqrt{\bar\sigma^2\zeta^4+2\bar\sigma\zeta\sqrt{1-\zeta^2}|A|_{12}}-\bar\sigma\zeta^2\right)+O(\xi).
$$
This completes the proof of the instability of Godunov/Roe space discretization
\end{proof}\\

\noindent
Let us now consider the stability of Lax Friedrichs type schemes. We will assume that $Q=\lambda_1 q\,{\rm Id}$ with $q=1/\lambda_1$ for the Lax Friedrichs scheme, $q=1/(2\lambda_1)$ for the modified one and $q=\rho(A)$ for the Rusanov scheme. It is an easy computation to show that
$$
\displaystyle
\Lambda_{\pm}(\xi,\delta x)=\zeta(\bar u\pm\sqrt{\bar c^2+\bar\sigma\xi^2})+i\frac{\xi\delta x}{2}q,\quad \zeta=\sqrt{1-\frac{\xi^2\delta x^2}{4}}.
$$
\noindent
\begin{proposition}
Assume $Q=\displaystyle\lambda_1q{\rm Id}$, then the (BE) and $(\Theta S)$ time discretization are unconditionally stable for all $\Theta\geq 1/2$. If $\Theta<1/2$, the $(\Theta S)$ scheme is stable under the condition.
$$
\displaystyle
\left((1-s)\left(\bar u\pm\sqrt{\bar c^2+\frac{4\bar\sigma s}{\delta x^2}}\right)^2+s q^2\right)\frac{\lambda_1}{q}\leq\frac{2}{1-2\Theta},\quad\forall s\in\left[0, 1\right].
$$
\end{proposition}

\noindent
One can derive a, simpler, sufficient condition of stability: indeed, it is easily seen that the above condition is satisfied if
$$
\displaystyle
\max\left(\left(|\bar u|+\sqrt{\bar c^2+\frac{4\bar\sigma}{\delta x^2}}\right)^2, \: q^2\right)\frac{\lambda_1}{q}\leq \frac{2}{1-2\Theta}
$$

\begin{corollary}
\noindent
The Lax Friedrichs scheme $q=1/\lambda_1$ is stable if
$$
\displaystyle
(|\bar u|+|\bar c|)\lambda_1+2\sqrt{\bar\sigma}\lambda_2\leq\sqrt{\frac{2}{(1-2\Theta)}}.
$$
The modified Lax Friedrichs scheme $q=1/(2\lambda_1)$ is stable if
$$
\displaystyle
(|\bar u|+|\bar c|)\lambda_1+2\sqrt{\bar\sigma}\lambda_2\leq\sqrt{\frac{1}{(1-2\Theta)}}.
$$
The Rusanov scheme $q=\rho(A)$ is stable if
$$
\displaystyle
\max\left(\left(|\bar u|\delta x+\sqrt{\bar c^2\delta x^2+4\bar\sigma}\right)^2\lambda_3, \rho(A)^2\lambda_1\right)\leq \frac{2\rho(A)}{1-2\Theta}.
$$
 \end{corollary}

One can show formally that the (CFL) condition $\delta t=O(\delta x^2)$ for Lax Friedrichs scheme is sharp. 
The wave speeds of (\ref{ek-lin}) are $s(\xi)=\bar{u}\pm\sqrt{\bar{c}^{2}+\bar{\sigma}\xi^{2}}$. Hence, in the 
limit $|\xi|\to\infty$, one has $s(\xi)\leq\bar{C}|\xi|$.
Heuristically, it is necessary for a numerical scheme to be stable
that the domain of dependence of the numerical solution contains the
domain of dependence of the exact solution. This condition reads
$s(\xi)\,\delta t/\delta x<1$. On a spatial grid with stepsize $\delta x$,
one has $s(\xi)\leq\bar{C}/\delta x$ since the largest wavenumber is
$O(1/\delta x)$. As a consequence, one obtains a (formal) CFL condition
$\bar{C}\delta t/\delta x^{2}<1$ which is precisely the CFL 
condition for the Lax Friedrichs scheme. The (CFL) found for the Rusanov scheme shows that this condition is not sufficient.\\

Note that if $\Theta=1/2$ (Crank Nicolson scheme), one can choose $Q=0$ and consequently a spatial centered scheme. This
corresponds to the numerical schemes used for the practical simulation
of thin film flows down an inclined plane in the presence of surface
tension \cite{KRSV}.

\subsection{Second order accurate schemes}

Hereafter, we consider second order accurate schemes. For the time discretization, we consider the
(second order) Runge Kutta and the Crank Nicolson methods. 
We discretize (\ref{ek-lin}) in space by using a MUSCL 
scheme \cite{CoL, VL} for the first order differential operator without nonlinear
monotony correction of the slope (it does not operate in the smooth
monotone area of the solution), and centered approximation of third
order differential terms: 
%
%{\setlength{\arraycolsep}{1pt}
%\[
%{\displaystyle \frac{dv_{j}}{dt}=-\frac{\triangle_{+}g_{LF}\left(v_{j-1/2,-},v_{j-1/2,+}\right)}{\delta x}+\frac{B}{2\delta x^{3}}\left(v_{j+2}-2v_{j+1}+2v_{j-1}-v_{j-2}\right)}
%\]
%where $v_{j-1/2,-}=v_{j-1}+\frac{d_{j-1}}{2},\; v_{j-1/2,+}=v_{j}-\frac{d_{j}}{2}$
%, $d_{j}$ is the local increment of $v$ given (without monotonicity
%correction) by the centered formula $d_{j}=\frac{v_{j+1}-v_{j-1}}{2}$
%: 
{\setlength{\arraycolsep}{1pt} 
\begin{eqnarray}
\displaystyle
\frac{dv_{j}}{dt}&=&\frac{A\left(v_{j+2}-6v_{j+1}+6v_{j-1}-v_{j-2}\right)}{8\delta x}\nonumber\\
\displaystyle
&+&\frac{Q}{2\delta t}\frac{\left(-v_{j+2}+4v_{j+1}-6v_{j}+4v_{j-1}-v_{j-2}\right)}{8} \nonumber \\
\displaystyle 
&+&\frac{B}{2\delta x^{3}}\left(v_{j+2}-2v_{j+1}+2v_{j-1}-v_{j-2}\right) \label{EK_muscl}
\end{eqnarray}
} with $j\in\mathbb{Z}$. In Fourier variables, the equation (\ref{EK_muscl}) now reads
%in the form $v_{j}(t)=e^{-ij\xi}\hat{v}(t)$: {one
%finds
\begin{equation}\label{MUSCL-Four}
\displaystyle
\frac{d\hat v}{dt}=i\xi\mathcal{M}(\xi,\delta x)\hat v,\quad \mathcal{M}(\xi,\delta x)=\zeta\left((1+\frac{\xi^2\delta x^2}{4})A+\xi^2B\right)+i\xi\left(\frac{\xi^2\delta x^4}{16\delta t}Q\right).
\end{equation}
\noindent
In what follows, we only consider second order accurate time discretization. First, the second order accurate Runge Kutta time discretization reads
\begin{equation}\label{muscl-four-rk}
\displaystyle
\hat v^{n+1}=\left({\rm Id}+i\xi\delta t\mathcal{M}(\xi,\delta x)-\frac{\xi^2\delta t^2}{2}\mathcal{M}(\xi,\delta x)^2\right)\hat v^n.
\end{equation}
\noindent
Assume $Q=\lambda_1q{\rm Id}$: the eigenvalues $\Lambda_{\pm}(\xi,\delta x)$ are given by
$$
\displaystyle
\Lambda_{\pm}(\xi,\delta x)=\zeta(2-\zeta^2)\left(\bar u\pm\sqrt{\bar c^2+\frac{\bar\sigma}{2-\zeta^2}\xi^2}\right)+iq\frac{\xi\delta x}{4}(1-\zeta^2).
$$

%%%%%%%%%%%%%%%%%%%%%%%%%%%%%%%%%%%%%%%%%%%%%
%\[
%{\displaystyle \delta t\frac{d\hat{v}}{dt}=\mathcal{M}(\xi,Q,\lambda_{1},\lambda_{3})\hat{v},}
%\]
%with
%\[
%\begin{array}{ll}
%\displaystyle 
%\mathcal{M}(\xi,Q,\lambda_{1},\lambda_{3})=-\frac{Q}{2}(1-\cos(\xi))^{2}Id+i\sin(\xi)M_{0}(\xi,\lambda_{1},\lambda_{3}),\:\\
%\displaystyle
%M_{0}(\xi,\lambda_{1},\lambda_{3})=\frac{\lambda_{1}}{2}A\left(3-\cos(\xi)\right)+2\lambda_{3}B\left(1-\cos(\xi)\right)
%\end{array}
%\]

%\noindent {Then, the second order accurate Runge-Kutta scheme is written as 
%\[
%{\displaystyle \hat{v}^{n+1}=\left({\rm Id}+\mathcal{M}(\xi,Q,\lambda_{1},\lambda_{3})+\frac{1}{2}\mathcal{M}(\xi,Q,\lambda_{1},\lambda_{3})^{2}\right)\hat{v}^{n},}
%\]
%}

%\noindent The eigenvalues of $\mathcal{M}(\xi,Q,\lambda_{1},\lambda_{3})$
%are  {$\lambda_{\pm}=-\frac{Q}{2}(1-\cos(\xi))^{2}-i\lambda_{1}\sin(\xi)\gamma_{\pm}$
%with 
%\[
%{\displaystyle \gamma_{\pm}=\frac{3-\cos(\xi)}{2}\left(\bar{u}\pm\sqrt{\bar{c}^{2}+\frac{4\bar{\sigma}\lambda_{3}(1-\cos(\xi))}{\lambda_{1}(3-\cos(\xi))}}\right)}
%\]
%}In order to simplify notations, we introduce $\tilde{Q}=\frac{Q}{2}(1-\cos(\xi))^{2}$
%and {${\Lambda_{1}=\lambda_{1}\gamma_{\pm}\sin(\xi)}$. }

\noindent Then, the Runge Kutta scheme is stable if and only if
\[
{\displaystyle {\left| 1+i\xi\delta t\Lambda_{\pm}(\xi,\delta x)-\frac{\xi^2\delta t^2}{2}\Lambda_{\pm}(\xi,\delta x)^2 \right|\leq1}}
\]
%\noindent or equivalently ${\displaystyle\frac{\Lambda_{1}^{2}}{2}\leq\tilde{Q}-\frac{\tilde{Q}^{2}}{2}+\sqrt{2\tilde{Q}-\tilde{Q}^{2}}}$.
%\noindent With $\tilde{Q}=\frac{1}{2}Qx$ and $\displaystyle\Lambda_{1}^{2}=\lambda_{1}^{2}x\left(2-x\right)\left(\frac{2+x}{2}\right)^{2}\left(\bar{u}+\sqrt{\bar{c}^{2}+\frac{4\bar{\sigma}\lambda_{3}x}{\lambda_{1}(2+x)}}\right)^{2}$.\\
%\noindent
%As a result, the explicit second order accurate
%scheme (MUSCL) is stable only if 
%\[
%{\lambda_{1}^{2}{\displaystyle \left(\bar{u}+\sqrt{\bar{c}^{2}+\frac{4\bar{\sigma}x}{\delta x^{2}\left(2+x\right)}}\right)^{2}\leq4\frac{{\displaystyle \sqrt{4Q-Q^{2}x^{2}}+Qx-\frac{Q^{2}x^{3}}{4}}}{(2-x)(2+x)^{2}}}}
%\]
%\noindent and 
The proof of the following proposition is straightforward

\begin{proposition} The second order accurate scheme with MUSCL type
discretization in space and Runge Kutta time discretization is stable
if and only if $0<q\lambda_1\leq 2$ and, for all $s=1-\zeta^2\in[0, 1]$
\begin{equation}
\displaystyle
\left(|\bar u|+\sqrt{\bar c^2+\frac{s}{1+s}\frac{4\bar\sigma}{\delta x^2}}\right)^2\lambda_1^2\leq\frac{q\lambda_1s(2-q\lambda_1s^2)+2\sqrt{2q\lambda_1-(q\lambda_1)^2s^2}}{4(1+s)(1-s)}
\end{equation} 
\end{proposition}

\noindent
From this proposition, we deduce the following simplified (CFL) conditions:

\begin{corollary} The classical Lax Friedrichs scheme with MUSCL
space and Runge Kutta time discretization is stable in the Von Neumann sense if
\[
{\lambda_{1}{\displaystyle \left(\bar{u}+\sqrt{\bar{c}^{2}+\frac{2\bar{\sigma}}{\delta x^{2}}}\right)\leq1}}
\]
{When $\delta x$ is sufficiently small one finds
the following condition:
\[
\frac{\delta t}{\delta x{}^{\frac{7}{3}}}\leq\left(\frac{\sqrt{\left(\bar{u}+\bar{c}\right)}}{\bar{\sigma}}\right)^{\frac{2}{3}}\frac{7^{\frac{7}{6}}\sqrt{3}}{24}+O\left(\delta x\right)+O\left(\sqrt{\lambda_{1}}\right)
\]
for the Rusanov scheme with MUSCL space discretization} \end{corollary}\\

\noindent
Note that we get an improved (CFL) condition $\delta t=O(\delta x^{7/3})$ for the Rusanov scheme that is almost sharp in comparison to
 first order accurate schemes. We finish this section by checking the stability of the Crank Nicolson scheme ($\Theta$-scheme with
$\Theta=1/2$). \\

 \begin{proposition}
The difference approximation with Crank Nicolson type time discretization
and second order accurate in space (MUSCL with Lax Friedrichs fluxes)
is stable for all $Q\geq0$. \end{proposition}

\begin{proof} In Fourier variables, the Crank Nicolson scheme for \ref{EK_muscl} reads
\[
{\displaystyle {\hat{v}^{n+1}=\left(1-\frac{i\xi\delta t}{2}\mathcal{M}(\xi,\delta x)\right)^{-1}\left(1+\frac{i\xi\delta t}{2}\mathcal{M}(\xi,\delta x)\right)\hat{v}^{n}.}}
\]
It is stable if and only if
$$
\displaystyle
\left|1+\frac{i\xi\delta t}{2}\Lambda_{\pm}(\xi,\delta x)\right|\leq\left|1-\frac{i\xi\delta t}{2}\Lambda_{\pm}(\xi,\delta x)\right|.
$$
It is easily seen that this condition is equivalent to $\Im(\xi\Lambda_{\pm}(\xi,\delta x))\geq 0$ which obviously holds true for any $q\geq 0$, and this concludes the proof
of the proposition.\end{proof}

%%%%%%%%%%%%%%%%%%%%%%%%%%%%%%%%%%%%%%%%%%
%%%%%%%%%%%%%%%%%%%%%SECTION 3%%%%%%%%%%%%%%%

\section{\label{sec3} Entropy stability of difference approximations}

In this section, we study the entropy stability of difference approximations
for Euler Korteweg equations (\ref{E-Kg}). In order to simplify the discussion,
we assume that (\ref{E-Kg}) is set on a bounded interval $[0, L]$ with periodic 
boundary conditions. Recall that $(\rho,u)$
solution of (\ref{E-Kg}) satisfies the energy estimate 
\begin{equation}\label{ek-en-nl}
{\displaystyle \partial_{t}\int_{\mathcal{D}}\rho\frac{u^{2}}{2}+F(\rho)+\kappa(\rho)\frac{(\partial_{x}\rho)^{2}}{2}dx=0,}
\end{equation}
with $\mathcal{D}=\mathbb{R}/L\mathbb{Z}$. The surface
tension plays a significant role in the energy estimate and the previous
section illustrates that it is a non trivial task to obtain a numerical
scheme which conserves or, at least, dissipates the energy, even at
the linearized level.

In this section, we introduce a new unknown $w=\sqrt{\kappa(\rho)}\partial_{x}\rho/\sqrt{\rho}$
and derive an evolution equation for $w$. The system of evolution
PDEs for $(\rho,\rho u,\rho w)$ is made of a first order hyperbolic
part perturbed by a second order anti dissipative term. This latter
term is discretized by centered finite differences. We show that any
entropy dissipative schemes for the hyperbolic part (in the sense
defined by Tadmor in \cite{T}), provides an entropy dissipative scheme
for the ``augmented'' Euler Korteweg system. 

In addition, we introduce an alternative discretization of (\ref{E-Kg})  
by writing this system as a Hamiltonian system of PDEs. By discretizing
the Hamiltonian and writing the associated Hamiltonian system of ODEs,
we find a consistent semi discrete scheme that is naturally entropy 
conservative.

%Then we consider fully
%discrete schemes and show that (FE) time discretization is entropy
%stable under a suitable CFL condition which is consistent with the
%one derived in the previous section whereas (BE) time discretization
%is always stable.

\subsection{Extended formulation of the Euler-Korteweg system}

We start from the system (\ref{E-Kg}). Following \cite{BDD}, we introduce $w=\sqrt{\kappa(\rho)}\partial_{x}\rho/\sqrt{\rho}$.
One finds: 
{\setlength{\arraycolsep}{1pt}
\begin{eqnarray}
{\displaystyle \partial_{t}\rho+\partial_{x}(\rho u)} & = & 0,\label{EK_nls1}\\
{\displaystyle \partial_{t}(\rho u)+\partial_{x}(\rho u^{2}+P(\rho))} & = & \partial_{x}\left(\mu(\rho)\partial_{x}w\right),\label{EK_nls2}\\
{\displaystyle \partial_{t}(\rho w)+\partial_{x}(\rho u\, w)} & = & -\partial_{x}\left(\mu(\rho)\partial_{x}u\right),\label{EK_nls3}
\end{eqnarray}
} with $\mu(\rho)=\rho^{3/2}\sqrt{\kappa(\rho)}$. The equation (\ref{EK_nls3})
is derived by multiplying (\ref{EK_nls1}) by $\sqrt{\rho\kappa(\rho)}$ and deriving the resulting equation with respect to $x$. Let us set  $v=(\rho, \rho u,\rho w)^{T}$ and
$f(v)=(\rho u,\rho u^{2}+P(\rho),\rho uw)^{T}$: the system (\ref{EK_nls1}-\ref{EK_nls3}) now
reads 
\begin{equation}
{\displaystyle \partial_{t}v+\partial_{x}f(v)=\partial_{x}\left(B(\rho)\partial_{x}\, z\right),}\label{EK_sys}
\end{equation}
where $B(\rho)$ denotes the skew-symmetric matrix
\[
{\displaystyle B(\rho)=\left(\begin{array}{ccc}
0 & 0 & 0\\
0 & 0 & \mu(\rho)\\
0 & -\mu(\rho) & 0
\end{array}\right),}
\]
and $z=\nabla_v U(v)$ with $\displaystyle U(v)=\rho\frac{u^2+w^2}{2}+F(\rho)$. Note that we performed in fact a {\it reduction of order} of the Euler Korteweg system: the extended system only contains second order derivatives with respect to $x$. It is important to note that {\it the operator $B(\rho)$ is skew symmetric with respect to $z$ which are nothing but the entropy variables}. This strategy is rather different from the one found in \cite{YS} and \cite{HR} for generalized Korteweg de Vries equations and $p$ system with surface tension: in these papers, the equations are written as a first order system of ODEs with respect to $x$ and a local discontinuous Galerkin method is used to discretize equations. Though the analysis is rather delicate to get entropy stable schemes and it is only proved at the semi discretized level. We prove here that our strategy extends rather easily to the fully discrete problem and involves only classical schemes for the first order part.

The first order part of (\ref{EK_sys}) ($B=0$)
admits an entropy-entropy flux pair $(U,G_0)$ with $G_0(v)=u\left(U(v)+P(\rho)\right)$
%\[
%{\displaystyle U(v)=F(\rho)+\rho\frac{u^{2}+w^{2}}{2},\quad G_0(v)=u\left(U(v)+P(\rho)\right)}
%\]
whereas the extended system (\ref{EK_sys}) admits
an additional conservation law 
\begin{equation}\label{cons-law}
\displaystyle \partial_{t}U(v)+\partial_{x}G(v,\partial_x u, \partial_x w)=0,\quad  G(v,\partial_x u,\partial_x w)=G_0(v)-\mu(\rho)(u\partial_{x}w-w\partial_{x}u).
\end{equation}

\noindent 
We consider difference approximations of (\ref{EK_sys})
in the conservative form

\begin{equation}
{\displaystyle \frac{d}{dt}v_{j}(t)+\frac{f_{j+\frac{1}{2}}-f_{j-\frac{1}{2}}}{\delta x}=\frac{1}{\delta x^{2}}\left(B(\rho_{j+\frac{1}{2}})\left(z_{j+1}-z_{j}\right)-B(\rho_{j-\frac{1}{2}})\left(z_{j}-z_{j-1}\right)\right).}\label{EKnls_sd}
\end{equation}

\noindent Following the terminology of \cite{T}, we enquire
when the difference schemes (\ref{EKnls_sd}) are \textit{entropy
stable} in the sense that there exists a numerical flux $\mathcal{G}_{j+\frac{1}{2}}$,
that is consistent with the full flux $G$,
so that 
\begin{equation}
{\displaystyle \frac{d}{dt}U(v_{j}(t))+\frac{\mathcal{G}_{j+\frac{1}{2}}-\mathcal{G}_{j-\frac{1}{2}}}{\delta x}\leq0.}\label{def_ent}
\end{equation}
The difference approximation (\ref{EKnls_sd}) is \textit{entropy
conservative} if the inequality in (\ref{def_ent}) is an equality.
Note that any entropy-stable scheme satisfies the entropy inequality
of the original system (\ref{E-Kg}) in a weaker sense since $w_{j}(t)$
is an approximation of $\sqrt{\kappa(\rho)}\partial_{x}\rho/\sqrt{\rho}$
at point $x_{j}=j\,\delta x$. In the last part of the section,
we will use the Hamiltonian structure of (\ref{E-Kg}) to obtain a semi-discrete
entropy conservative scheme. In what follows, we prove the following proposition.\\

 \begin{proposition}\label{prop31}
Let us consider the finite difference scheme 
\begin{equation}
{\displaystyle \frac{d}{dt}v_{j}(t)+\frac{f_{j+\frac{1}{2}}-f_{j-\frac{1}{2}}}{\delta x}=0,}\label{sd_hyp}
\end{equation}
which is a semi discretization of (\ref{EK_sys}) with $B=0$ and
is entropy stable. That is, there exists a numerical flux $\mathcal{G}_{0,j+\frac{1}{2}}$ which is consistent with $G_0$ so that
$$
{\displaystyle \frac{d}{dt}U(v_{j}(t))+\frac{\mathcal{G}_{0,j+\frac{1}{2}}-\mathcal{G}_{0,j-\frac{1}{2}}}{\delta x}\leq0.}
$$
Then the difference scheme (\ref{EKnls_sd}) is
entropy stable. \end{proposition}\\

\begin{proof} The difference approximation (\ref{sd_hyp}) is entropy
stable: there exists a numerical entropy flux $\mathcal{G}^0_{j+\frac{1}{2}}$ which is consistent with $G_0$
so that 
\[
{\displaystyle \nabla_{v}U(v_{j})^{T}\frac{f_{j+\frac{1}{2}}-f_{j-\frac{1}{2}}}{\delta x}=\frac{\mathcal{G}_{0,j+\frac{1}{2}}-\mathcal{G}_{0,j-\frac{1}{2}}}{\delta x}+\mathcal{R}_{j},}
\]
with $\mathcal{R}_{j}\geq0$ (see \cite{T} for more details). We multiply (\ref{EKnls_sd}) by $\nabla_{v}U(v_{j})^{T}$:
 {\setlength{\arraycolsep}{1pt} 
\begin{eqnarray*}
\displaystyle \frac{d}{dt}U(v_{j})&+&\frac{\mathcal{G}_{0,j+\frac{1}{2}}-\mathcal{G}_{0,j-\frac{1}{2}}}{\delta x}+\mathcal{R}_{j} \nonumber\\
\displaystyle
& = & \frac{\nabla_vU(v_{j})^{T}}{\delta x^{2}}\left(B(\rho_{j+\frac{1}{2}})\big(z_{j+1}-z_{j}\big)-B(\rho_{j-\frac{1}{2}})\big(z_{j}-z_{j-1}\big)\right):=\mathcal{K}_{j},
\end{eqnarray*}
} 
We focus on the capillary term $\mathcal{K}_{j}$: it is
written as {\setlength{\arraycolsep}{1pt} 
\begin{eqnarray}
{\displaystyle \delta x^{2}\mathcal{K}_{j}} & = & u_{j}\left(\mu_{j+\frac{1}{2}}(w_{j+1}-w_{j})-\mu_{j-\frac{1}{2}}(w_{j}-w_{j-1})\right)\nonumber \\
{\displaystyle } &  & -w_{j}\left(\mu_{j+\frac{1}{2}}(u_{j+1}-u_{j})-\mu_{j-\frac{1}{2}}(u_{j}-u_{j-1})\right)\nonumber \\
{\displaystyle } & = & \mu_{j+\frac{1}{2}}\left(u_{j}w_{j+1}-u_{j+1}w_{j}\right)-\mu_{j-\frac{1}{2}}\left(u_{j-1}w_{j}-u_{j}w_{j-1}\right).\nonumber 
\end{eqnarray}
} 
Now, we introduce the entropy flux $\mathcal{G}_{j+\frac{1}{2}}$:
\[
{\displaystyle \mathcal{G}_{j+\frac{1}{2}}=\mathcal{G}_{0,j+\frac{1}{2}}-\mu_{j+\frac{1}{2}}\frac{u_{j}w_{j+1}-u_{j+1}w_{j}}{\delta x}.}
\]
This numerical entropy flux is clearly consistent with the continuous one given by
\[
{\displaystyle {G}(v,\partial_{x}u,\partial_{x}w)=G_0(v)-\mu(\rho)\left(u\partial_{x}w-w\partial_{x}u\right).}
\]
%provided  $\mathcal{F}_{j+\frac{1}{2}}$ is consistent
%with the entropy flux $G(v)$ of the hyperbolic part of (\ref{EK_nls1}-\ref{EK_nls3}).
Moreover, we have the following semi discrete entropy estimate 
\[
{\displaystyle \frac{d}{dt}U(v_{j})+\frac{\left(\mathcal{G}_{j+\frac{1}{2}}-\mathcal{G}_{j-\frac{1}{2}}\right)}{\delta x}=-\mathcal{R}_{j}\leq0.}
\]
This completes the proof of the proposition. \end{proof}\\

\noindent By applying proposition \ref{prop31}, 
%and considering the analysis of various entropy stable schemes found in \cite{T},
one finds that many of the classical three points (first order) schemes
(Rusanov, Lax Friedrichs and Harten-Lax-van Leer schemes)
provide natural entropy stable schemes for the augmented system (\ref{EK_nls1}-\ref{EK_nls3}). 
The Roe and Godunov schemes are stable as well for the semi discretized problem and stable for the
fully discretized scheme with Backward Euler time discretization. We checked the Von Neumann stability
of the Forward Euler time discretization together with Roe/Godunov space discretization: one can prove 
that it is unstable in the Von Neumann sense. For application purposes, we check the entropy stability of 
\textit{fully discrete schemes} associated to Lax-Friedrichs type space discretizations..

%%%TO DO: Quid du schŽma de Roe/Godunov dans cette configuration?????%%%%

\subsection{Entropy stability of fully-discrete schemes}

In this section, we consider the entropy stability of fully discrete schemes. We restrict our discussion to first order Forward/Backward Euler schemes which read, respectively:
{\setlength\arraycolsep{1pt}
\begin{eqnarray}
\displaystyle 
v_{j}^{n+1}-v_{j}^{n}&+&\lambda_{1}\left(f_{j+\frac{1}{2}}^{n+1}-f_{j-\frac{1}{2}}^{n+1}\right)\nonumber\\
\displaystyle
&=&\lambda_{2}\left(B(\rho_{j+\frac{1}{2}}^{n+1})\left(z_{j+1}^{n+1}-z_{j}^{n+1}\right)-B(\rho_{j-\frac{1}{2}}^{n+1})\left(z_{j}^{n+1}-z_{j-1}^{n+1}\right)\right).\label{EKnls_bi}\\
\displaystyle
v_{j}^{n+1}-v_{j}^{n}&+&\lambda_{1}\left(f_{j+\frac{1}{2}}^{n}-f_{j-\frac{1}{2}}^{n}\right)\nonumber\\
\displaystyle
&=&\lambda_{2}\left(B(\rho_{j+\frac{1}{2}}^{n})\left(z_{j+1}^{n}-z_{j}^{n}\right)-B(\rho_{j-\frac{1}{2}}^{n})\left(z_{j}^{n}-z_{j-1}^{n}\right)\right).\label{EKnls_fe}
\end{eqnarray}
}

\noindent We first prove the entropy stability of the implicit backward
Euler time discretization.

\begin{proposition} Assume that the semi discretized scheme 
\begin{equation}
{\displaystyle \frac{d}{dt}v_{j}(t)+\frac{f_{j+\frac{1}{2}}-f_{j-\frac{1}{2}}}{dx}=0,}\label{sd_hyp2}
\end{equation}
 is an entropy stable approximation of (\ref{EK_sys}) with $B=0$. That is, there exists an entropy flux $\mathcal{G}_{0,j+\frac{1}{2}}$ which is consistent with $G_0$ so that
 $$
 \displaystyle
 \frac{d}{dt}U(v_j(t))+\frac{\mathcal{G}_{0,j+\frac{1}{2}}-\mathcal{G}_{0,j-\frac{1}{2}}}{\delta x}\leq 0.
 $$
 Then the scheme (\ref{EKnls_bi}) is
 entropy stable: define $\mathcal{G}_{j+\frac{1}{2}}^{n}$ as
 $$
 {\displaystyle \mathcal{G}_{j+\frac{1}{2}}^{n}=\mathcal{G}_{0,j+\frac{1}{2}}(v^{n+1})-\mu_{j+\frac{1}{2}}^{n+1}\frac{u_{j}^{n+1}w_{j+1}^{n+1}-u_{j+1}^{n+1}w_{j}^{n+1}}{\delta x}}
 $$
\noindent
Then, the following discrete entropy inequality is satisfied: 
\begin{equation}
{\displaystyle U(v_{j}^{n+1})-U(v_{j}^{n})+\frac{\delta t}{\delta x}(\mathcal{G}_{j+\frac{1}{2}}^{n}-\mathcal{G}_{j-\frac{1}{2}}^{n})\leq0,\forall j\in\mathbb{Z},\quad\forall n\in\mathbb{N}.}\label{ent_dis}
\end{equation}
\end{proposition}

\noindent
\begin{proof} Since $\rho F''(\rho)=P'(\rho)>0$, the entropy $U$ is a convex function of $v$ as long as $\rho>0$. Then, one has 
\begin{equation}
{\displaystyle U(v_{j}^{n+1})\leq U(v_{j}^{n})+\nabla_vU(v_{j}^{n+1})^{T}(v_{j}^{n+1}-v_{j}^{n}).}\label{ineq_ent}
\end{equation}
The semi-discrete scheme (\ref{sd_hyp2}) is entropy stable so that (see \cite{T} for more details)
\[
{\displaystyle U_{v}(v_{j}^{n+1})^{T}(f_{j+\frac{1}{2}}^{n+1}-f_{j-\frac{1}{2}}^{n+1})=\mathcal{G}_{0,j+\frac{1}{2}}(v^{n+1})-\mathcal{G}_{0,j-\frac{1}{2}}(v^{n+1})+\delta x\mathcal{R}_{j}^{n},}
\]
for some $\mathcal{R}_{j}^{n}\geq0$. Moreover, one has {\setlength{\arraycolsep}{1pt}
\begin{eqnarray}
{\displaystyle U_{v}(v_{j}^{n+1})^{T}\Big(B(\rho_{j+\frac{1}{2}}^{n+1})\big(z_{j+1}^{n+1}-z_{j}^{n+1}\big)} & - & B(\rho_{j-\frac{1}{2}}^{n+1})\big(z_{j}^{n+1}-z_{j-1}^{n+1}\big)\Big)\nonumber \\
{\displaystyle =\mu_{j+\frac{1}{2}}^{n+1}\left(u_{j}^{n+1}w_{j+1}^{n+1}-u_{j+1}^{n+1}w_{j}^{n+1}\right)} & - & \mu_{j-\frac{1}{2}}^{n+1}\left(u_{j-1}^{n+1}w_{j}^{n+1}-u_{j}^{n+1}w_{j-1}^{n+1}\right).\label{cons-cap}
\end{eqnarray}
} Now we introduce the entropy flux $\mathcal{G}_{j+\frac{1}{2}}^{n}$:
\[
{\displaystyle \mathcal{G}_{j+\frac{1}{2}}^{n}=\mathcal{G}_{0,j+\frac{1}{2}}(v^{n+1})-\mu_{j+\frac{1}{2}}^{n+1}\frac{u_{j}^{n+1}w_{j+1}^{n+1}-u_{j+1}^{n+1}w_{j}^{n+1}}{\delta x}}
\]
 Then, by inserting (\ref{EKnls_bi}) into (\ref{ineq_ent}) and 
using the definition of $\mathcal{G}_{j+\frac{1}{2}}^{n}$, one obtains
\[
{\displaystyle U(v_{j}^{n+1})-U(v_{j}^{n})+\lambda_{1}(\mathcal{G}_{j+\frac{1}{2}}^{n}-\mathcal{G}_{j-\frac{1}{2}}^{n})\leq-\lambda_{1}\mathcal{R}_{j}^{n}\leq0.}
\]
This completes the proof of the proposition. \end{proof}\\

\noindent
Next, we consider the entropy stability of the explicit scheme (\ref{EKnls_fe}). We restrict our attention to the schemes with numerical fluxes in the form
which admit the viscosity form: 
%\begin{equation}
%{\displaystyle \frac{d}{dt}v_{j}+\frac{f(v_{j+1})-f(v_{j-1})}{2\delta x}=\frac{1}{2\delta x}\left(Q_{j+\frac{1}{2}}(\tilde{z}_{j+1}-\tilde{z}_{j})-Q_{j-\frac{1}{2}}(\tilde{z}_{j}-\tilde{z}_{j-1})\right),}\label{cons_visq}
%\end{equation}
%so that the flux $f_{j+\frac{1}{2}}$ reads 
\begin{equation}\label{cons_visq}
{\displaystyle f_{j+\frac{1}{2}}=\frac{f(v_{j+1})+f(v_{j})}{2}-\frac{1}{2}Q_{j+\frac{1}{2}}({z}_{j+1}-{z}_{j}).}
\end{equation}
The matrix $Q_{j+\frac{1}{2}}$ is a symmetric matrix whereas ${z}=\nabla_vU(v)$
represent the entropy variables. It is easily seen that $z_{2}=u$
and $z_{3}=w$. Here the conservative variables $v$
are considered as functions of the entropy variables: in particular
$v_{j}=v({z}_{j})$. In this setting, there exists a viscosity matrix $Q^*_{j+\frac{1}{2}}$ so that the scheme
\begin{equation}\label{cons-scheme}
\displaystyle
\frac{dv_j^*}{dt}+\frac{f^*_{j+\frac{1}{2}}-f^*_{j-\frac{1}{2}}}{\delta x}=0,
\end{equation}
\noindent
is exactly entropy conservative. More precisely, by setting
\[
{\displaystyle Q_{j+\frac{1}{2}}^{*}=\int_{-1/2}^{1/2}2\xi\, \nabla_z g\left(\frac{{z}_{j+1}+{z}_{j}}{2}+\xi({z}_{j+1}-{z}_{j})\right)d\xi},\quad g(z)=f(v(z)),
\]
\noindent
one proves that there exists a numerical flux $\mathcal{G}^*_{0,j+\frac{1}{2}}$ consistent with $G_0$ so that
$$
\displaystyle
\frac{d}{dt}U(v_j^*(t))+\frac{\mathcal{G}^*_{0,j+\frac{1}{2}}-\mathcal{G}^*_{0,j-\frac{1}{2}}}{\delta x}=0.
$$
\noindent
The classical Lax Friedrichs and
Rusanov scheme are particular cases of (\ref{cons_visq}). Indeed,
these schemes have the particular form 
\begin{equation}
{\displaystyle f_{j+\frac{1}{2}}=\frac{f(v_{j+1})+f(v_{j})}{2}-\frac{1}{2}p_{j+\frac{1}{2}}(v_{j+1}-v_{j}),}\label{cons_lf}
\end{equation}
with $p_{j+\frac{1}{2}}\geq0$. 
%Indeed, one
%has
%\[
%\left(\int_{0}^{1}v_{\tilde{z}}(\tilde{z}_{j}+\xi(\tilde{z}_{j+1}-\tilde{z}_{j}))d\xi\right)(\tilde{z}_{j+1}-\tilde{z}_{j})=v_{j+1}-v_{j}
%\]
%\[
%v_{\tilde{z}}=\left(\tilde{z}_{v}\right)^{-1}=\left(U_{vv}\right)^{-1}
%\]
%\noindent
%then $v_{\tilde{z}}=v_{\tilde{z}}^{T}$ , so that 
The flux (\ref{cons_lf}) is a particular case of (\ref{cons_visq}) by setting 
\[
{\displaystyle Q_{j+\frac{1}{2}}=p_{j+\frac{1}{2}}\int_{0}^{1}\nabla_zv({z}_{j}+\xi({z}_{j+1}-{z}_{j}))d\xi}
\]
 with $Q_{j+\frac{1}{2}}=Q_{j+\frac{1}{2}}^{T}$.  Following \cite{T}, (Corollary 5.1 p. 472-473), one can compare any conservative scheme 
$$
\displaystyle
\frac{dv_j}{dt}+\frac{f_{j+\frac{1}{2}}-f_{j-\frac{1}{2}}}{\delta x}=0,
$$
the flux $f_{j+\frac{1}{2}}$ being defined by  (\ref{cons_visq}), with the
entropy conservative scheme (\ref{cons-scheme}) through the relation: {\setlength{\arraycolsep}{1pt}
\begin{eqnarray}
&&{\displaystyle \langle \nabla_vU(v_j),f_{j+\frac{1}{2}}-f_{j-\frac{1}{2}}\rangle}=\mathcal{G}_{0,j+\frac{1}{2}}-\mathcal{G}_{0,j-\frac{1}{2}}\nonumber \\
{\displaystyle } & + & \frac{1}{4}\left(\langle({z}_{j+1}-{z}_{j}),D_{j+\frac{1}{2}}({z}_{j+1}-{z}_{j})\rangle+\langle({z}_{j}-{z}_{j-1}),D_{j-\frac{1}{2}}({z}_{j}-{z}_{j-1})\rangle\right),\label{num_diss}
\end{eqnarray}
}
\noindent with $D_{j+\frac{1}{2}}=Q_{j+\frac{1}{2}}-Q_{j+\frac{1}{2}}^{*}$ and $\mathcal{G}_{0,j+\frac{1}{2}}$
is a consistent entropy flux given by
$$
\displaystyle
\mathcal{G}_{0,j+\frac{1}{2}}=\left\langle\frac{ z_j+ z_{j+1}}{2}; f_{j+\frac{1}{2}}\right\rangle-\frac{1}{2}\left(\psi( z_j)+\psi(z_{j+1})\right),\quad \psi( z)=\langle z, g(z)\rangle-G_0(v( z)).
$$
\noindent
We prove the following proposition.\\

%%%%%%%%%%%%%%%%%%%%%%%
%%%%%%%%%%%%%%%%%%%%%%%%
%%%%%%%%%%%%%%%%%%%%%%%%
%%%%%%%%%%%%%%%%%%%%%%%%
%%%%%%%%%%%%%%%%%%%%%%%%
\begin{proposition}\label{prop_nl} The  finite difference scheme
\begin{equation}
\begin{array}{ll}
{\displaystyle v_{j}^{n+1}-v_{j}^{n}+\lambda_{1}(f^n_{j+\frac{1}{2}}-f^n_{j-\frac{1}{2}})=\lambda_{2}\left(B_{j+\frac{1}{2}}^{n}(z_{j+1}^{n}-z_{j}^{n})-B_{j-\frac{1}{2}}^{n}(z_{j}^{n}-z_{j-1}^{n})\right),}\\
{\displaystyle f_{j+\frac{1}{2}}^{n}=\frac{f(v_{j+1}^{n})+f(v_{j}^{n})}{2}-\frac{1}{2}Q_{j+\frac{1}{2}}^{n}(z_{j+1}^{n}-z_{j}^{n})}
\end{array}\label{sd_hyp3}
\end{equation}
is entropy stable, i.e. there exists a numerical entropy flux $\mathcal{G}_{j+\frac{1}{2}}^{n}$
so that 
\[
{\displaystyle U(v_{j}^{n+1})-U(v_{j}^{n})+\lambda_{1}(\mathcal{G}_{j+\frac{1}{2}}^{n}-\mathcal{G}_{j-\frac{1}{2}}^{n})\leq0,}
\]
under the following CFL condition 
\[
{\displaystyle M_{j}^{n}\left(\lambda_{1}N_{j+\frac{1}{2}}^{n}+\lambda_{2}\|B_{j+\frac{1}{2}}^{n}\|\right)^{2}\leq\lambda_{1}\min({\rm Sp}(D_{j+\frac{1}{2}}^{n})),}
\]
with $M_{j}^{n},N_{j+\frac{1}{2}}^{n}$ defined as 
\[
\begin{array}{ll}
{\displaystyle M_{j}^{n}=\sup_{\xi\in(0,1)}\|\nabla^2_vU(v_{j}^{n}+\xi(v_{j}^{n+1}-v_{j}^{n}))\|,}\\
{\displaystyle N_{j+\frac{1}{2}}^{n}=\int_{0}^{1}\|\nabla_z g\left({z}_{j}^{n}+\xi({z}_{j+1}^{n}-{z}_{j}^{n})\right)\|d\xi+\|Q_{j+\frac{1}{2}}^{n}\|.}
\end{array}
\]
\end{proposition}

\noindent 
\begin{proof} We first apply the Taylor Lagrange formula
to $U$: 
{\setlength{\arraycolsep}{1pt} 
\begin{eqnarray*}
{\displaystyle  U(v_{j}^{n+1})}&=&U(v_{j}^{n})+\nabla_{v}U(v_{j}^{n})^{T}(v_{j}^{n+1}-v_{j}^{n})\\
{\displaystyle } &+&\int_{0}^{1}(1-\xi)(v_{j}^{n+1}-v_{j}^{n})^{T}\nabla^2_vU(v_{j}^{n}+\xi(v_{j}^{n+1}-v_{j}^{n}))(v_{j}^{n+1}-v_{j}^{n})d\xi.
\end{eqnarray*}
} 
Then, by using (\ref{num_diss}) and a relation similar to (\ref{cons-cap}), one finds 
{\setlength{\arraycolsep}{1pt}
\begin{eqnarray}
\displaystyle && U(v_{j}^{n+1}) -  U(v_{j}^{n})+\lambda_{1}\left(\mathcal{G}_{j+\frac{1}{2}}^{n}-\mathcal{G}_{j-\frac{1}{2}}^{n}\right)=\nonumber \\
\displaystyle  &&\int_{0}^{1}(1-\xi)(v_{j}^{n+1}-v_{j}^{n})^{T}\nabla^2_vU(v_{j}^{n}+\xi(v_{j}^{n+1}-v_{j}^{n}))(v_{j}^{n+1}-v_{j}^{n})d\xi \nonumber \\
\displaystyle &&-\frac{\lambda_{1}}{4}\left(({z}_{j+1}^{n}-{z}_{j}^{n})^{T}D_{j+\frac{1}{2}}^{n}({z}_{j+1}^{n}-{z}_{j}^{n})+({z}_{j}^{n}-{z}_{j-1}^{n})^{T}D_{j-\frac{1}{2}}^{n}({z}_{j}^{n}-{z}_{j-1}^{n})\right).\label{en_bal}
\end{eqnarray}
} 
with 
$$
\displaystyle
\mathcal{G}_{j+\frac{1}{2}}^n=\mathcal{G}_{0,j+\frac{1}{2}}^n-\mu_{j+\frac{1}{2}}^n\frac{u_j^nw_{j+1}^n-u_{j+1}^nw_j^n}{\delta x}.
$$
The first term in the right hand side of (\ref{en_bal}) is positive
and corresponds to entropy production due to the forward explicit
Euler time discretization whereas the second term corresponds to entropy
dissipation due to the spatial discretization. Next, we estimate the
entropy production: in order to simplify notations, we set 
\[
{\displaystyle \mathcal{I}_{j}^{n}=\int_{0}^{1}(1-\xi)(v_{j}^{n+1}-v_{j}^{n})^{T}\nabla^2_vU(v_{j}^{n}+\xi(v_{j}^{n+1}-v_{j}^{n}))(v_{j}^{n+1}-v_{j}^{n})d\xi.}
\]
One has 
\[
{\displaystyle \mathcal{I}_{j}^{n}\leq\frac{1}{2}\sup_{\xi\in(0,1)}\|\nabla^2_vU(v_{j}^{n}+\xi(v_{j}^{n+1}-v_{j}^{n}))\|\|v_{j+1}^{n}-v_{j}^{n}\|^{2}:=\frac{1}{2}M_{j}^{n}\|v_{j+1}^{n}-v_{j}^{n}\|^{2}.}
\]
Next, we estimate $\|v_{j}^{n+1}-v_{j}^{n}\|$ by using (\ref{sd_hyp3}):
one finds 
\[
{\displaystyle \|v_{j}^{n+1}-v_{j}^{n}\|\leq\lambda_{1}\|f_{j+\frac{1}{2}}^{n}-f_{j-\frac{1}{2}}^{n}\|+\lambda_{2}\left(\|B_{j+\frac{1}{2}}^{n}\|\,\|z_{j+1}^{n}-z_{j}^{n}\|+\|B_{j-\frac{1}{2}}^{n}\|\|z_{j}^{n}-z_{j-1}^{n}\|\right).}
\]
On the other hand, one has
 {\setlength{\arraycolsep}{1pt} 
\begin{eqnarray*}
{\displaystyle f_{j+\frac{1}{2}}^{n}-f_{j-\frac{1}{2}}^{n}} & = & \left(\int_{0}^{1}\nabla_zg\left({z}_{j}^{n}+\xi({z}_{j+1}^{n}-{z}_{j}^{n})\right)d\xi-Q_{j+\frac{1}{2}}^{n}\right)({z}_{j+1}^{n}-{z}_{j}^{n})\\
{\displaystyle } & + & \left(\int_{0}^{1}\nabla_zg\left({z}_{j}^{n}+\xi({z}_{j}^{n}-{z}_{j-1}^{n})\right)d\xi+Q_{j-\frac{1}{2}}^{n}\right)({z}_{j}^{n}-{z}_{j-1}^{n}).
\end{eqnarray*}
} Then, by setting $N_{j+\frac{1}{2}}^{n}=\int_{0}^{1}\|\nabla_z g\left({z}_{j}^{n}+\xi({z}_{j+1}^{n}-{z}_{j}^{n})\right)\|d\xi+\|Q_{j+\frac{1}{2}}^{n}\|$,
one obtains 
\[
{\displaystyle \|f_{j+\frac{1}{2}}^{n}-f_{j-\frac{1}{2}}^{n}\|\leq N_{j+\frac{1}{2}}^{n}\|{z}_{j+1}^{n}-{z}_{j}^{n}\|+N_{j-\frac{1}{2}}^{n}\|{z}_{j}^{n}-{z}_{j-1}^{n}\|.}
\]
As a result, one finds that 
{\setlength{\arraycolsep}{1pt} 
\begin{eqnarray}
{\displaystyle \mathcal{I}_{j}^{n}} & \leq & M_{j}^{n}\left(\lambda_{1}N_{j+\frac{1}{2}}^{n}+\lambda_{2}\|B_{j+\frac{1}{2}}^{n}\|\right)^{2}\|{z}_{j+1}^{n}-{z}_{j}^{n}\|^{2}\nonumber \\
{\displaystyle } &  & +M_{j}^{n}\left(\lambda_{1}N_{j-\frac{1}{2}}^{n}+\lambda_{2}\|B_{j-\frac{1}{2}}^{n}\|\right)^{2}\|{z}_{j}^{n}-{z}_{j-1}^{n}\|^{2}.\label{crea_ent}
\end{eqnarray}
} Next, we set $\Gamma_{j+\frac{1}{2}}^{n}=\min\left({\rm Sp}(D_{j+1}^{n})\right)$.
Furthermore, we assume that 
\begin{equation}
{\displaystyle M_{j}^{n}(\lambda_{1}N_{j+\frac{1}{2}}^{n}+\lambda_{2}\|B_{j+\frac{1}{2}}^{n}\|)^{2}\leq\lambda_{1}\Gamma_{j+\frac{1}{2}}^{n}.}\label{cfl_nl}
\end{equation}
Then, by using (\ref{cfl_nl}) together with (\ref{crea_ent}) and
(\ref{en_bal}), one obtains entropy stability for the explicit forward
Euler time discretization 
\[
{\displaystyle U(v_{j}^{n+1})-U(v_{j}^{n})+\lambda_{1}(\mathcal{G}_{j+\frac{1}{2}}^{n}-\mathcal{G}_{j-\frac{1}{2}}^{n})\leq0.}
\]
This completes the proof of the proposition. \end{proof}\\

 Let us consider Lax Friedrichs schemes: by applying proposition
\ref{prop_nl}, we prove\\

\begin{corollary} Assume there exists $K>0$ so that $K^{-1}\leq M_{j}^{n}\leq K$ for all
$j,n$ and $p_{j+\frac{1}{2}}^{n}=\tilde{p}_{j+\frac{1}{2}}^{n}+\max\left(|{\rm Sp}(\nabla_vf(v_{j+1}^{n}))|,|{\rm Sp}(\nabla_{v}f(v_{j}^{n}))|\right).$ The Lax Friedrichs  scheme, $\tilde{p}_{j+\frac{1}{2}}^{n}={\displaystyle (2\lambda_1)^{-1}}$,
is entropy stable if
$
{\displaystyle K(\lambda_{1}M_{1}(K)+\lambda_{2}M_{2}(K))^{2}\leq 1/2,}
$
for some constants $M_{j}(K),j=1,2$. The Rusanov type scheme, $\tilde{p}_{j+\frac{1}{2}}^{n}=\rho>0$,
is entropy stable under the CFL condition 
$
{\displaystyle K(\lambda_{1}M_{1}(K)+\lambda_{2}M_{2}(K))^{2}\leq\rho\lambda_{1}.}
$
\end{corollary}\\

\noindent
{\it Remark:} The previous result states that the classical Lax Friedrichs
scheme is entropy stable if $\delta t=O(\delta x^{2})$ whereas the
Rusanov scheme is entropy stable only if $\delta t=O(\delta x^{3})$
which are the Von Neumann stability criterion found in section \ref{sec2}.

\subsection{A semi-discrete entropy conservative scheme}

In this section, we use the Hamiltonian structure of the Euler Korteweg
equations to construct an entropy conservative scheme. For that purpose,
we will write a semi discretized form of the Euler Korteweg system
which respects its Hamiltonian structure so that the entropy 
is automatically satisfied. We consider the Euler Korteweg equations with periodic boundary conditions
and, for $(\varrho,{\rm u})=(\rho_{i},u_{i})_{i=1,\cdots,N}$, we
introduce the discrete Hamiltonian 
\begin{equation}
{\displaystyle {\rm H}(\varrho,{\rm u})=\sum_{i=1}^{N}\rho_{i}\frac{u_{i}^{2}}{2}+F(\rho_{i})+\frac{1}{2}\kappa(\rho_{i})\left(\frac{\rho_{i+1}-\rho_{i}}{\delta x}\right)^{2}.}\label{ham_dis}
\end{equation}
We also introduce the symmetric matrix $J$ 
\[
{\displaystyle J=\left(\begin{array}{cc}
0 & -{\rm I}_{N}\\
-{\rm I}_{N} & 0
\end{array}\right),}
\]
and the difference operator $D$, defined in the space of $N$-periodic
sequences in $\mathbb{R}^{N}$ as $Du_{i}=\displaystyle\frac{u_{i+1}-u_{i-1}}{2\delta x}$
(the associated matrix $D\in M_{N}(\mathbb{R})$ is skew symmetric).
Then, we introduce the Hamiltonian system 
\begin{equation}\label{schem_ham}
{\displaystyle \frac{d}{dt}\left(\begin{array}{c}
\varrho\\
{\rm u}
\end{array}\right)=J\left(\begin{array}{c}
D\nabla_{\varrho}{\rm H}(\varrho,{\rm u})\\
D\nabla_{{\rm u}}{\rm H}(\varrho,{\rm u})
\end{array}\right).}
\end{equation}
%More precisely, the discrete Hamiltonian system reads {\setlength{\arraycolsep}{1pt}
%\begin{eqnarray*}
%{\displaystyle \frac{d\rho_{j}}{dt}} & + & \frac{(\rho u)_{j+1}-(\rho u)_{j-1}}{2\delta x}=0,\\
%{\displaystyle \frac{du_{j}}{dt}} & + & \frac{u_{j+1}^{2}-u_{j-1}^{2}}{4\delta x}+\frac{F'(\rho_{j+1})-F'(\rho_{j-1})}{2\delta x}\\
%{\displaystyle } & + & \frac{1}{2\delta x}\left(\frac{\kappa'(\rho_{j+1})}{2}\left(\frac{\rho_{j+2}-\rho_{j+1}}{\delta x}\right)^{2}-\frac{\kappa'(\rho_{j-1})}{2}\left(\frac{\rho_{j}-\rho_{j-1}}{\delta x}\right)^{2}\right)\\
%{\displaystyle } & - & \frac{1}{2\delta x^{2}}\left(\kappa(\rho_{j+1})\frac{\rho_{j+2}-\rho_{j+1}}{\delta x}-\kappa(\rho_{j})\frac{\rho_{j+1}-\rho_{j}}{\delta x}\right)\\
%{\displaystyle } & + & \frac{1}{2\delta x^{2}}\left(\kappa(\rho_{j-1})\frac{\rho_{j}-\rho_{j-1}}{\delta x}-\kappa(\rho_{j-2})\frac{\rho_{j-1}-\rho_{j-2}}{\delta x}\right)=0.
%\end{eqnarray*}
%} 
%By construction, one has {\setlength{\arraycolsep}{1pt} 
%\begin{eqnarray*}
%{\displaystyle \frac{d}{dt}{\rm H}(\varrho,{\rm u})} & = & \nabla_{\varrho}{\rm H}(\varrho,{\rm u})^{T}\frac{d\varrho}{dt}+\nabla_{{\rm u}}{\rm H}(\varrho,{\rm u})^{T}\frac{d{\rm u}}{dt}\\
%{\displaystyle } & = & \nabla_{\varrho}{\rm H}(\varrho,{\rm u})^{T}\, D\,\nabla_{{\rm u}}{\rm H}(\varrho,{\rm u})+\nabla_{{\rm u}}{\rm H}(\varrho,{\rm u})^{T}\, D\,\nabla_{\varrho}{\rm H}(\varrho,{\rm u})\\
%{\displaystyle } & = & \nabla_{\varrho}{\rm H}(\varrho,{\rm u})^{T}\,(D+D^{T})\,\nabla_{{\rm u}}{\rm H}(\varrho,{\rm u})=0.
%\end{eqnarray*}
%}
\noindent
It is clearly a consistent and first order discretization of the original system (\ref{E-Kh}). Due to the loss of translation invariance, the momentum
is not exactly preserved. Though, we do not expect formation of discontinuities and in return we expect convergence of the momentum.
Anyway, by construction, one has 
$\displaystyle {\rm H}(\varrho,{\rm u})(t) ={\rm H}(\varrho,{\rm u})(0).$
In contrast to the entropy conservative schemes proposed by Tadmor \cite{T}, this scheme has no numerical viscosity
which makes possible the numerical simulation of dispersive shock waves \cite{E,EGK}.
Moreover one can go one step further and derive naturally higher order entropy conservative scheme like in  \cite{CL}, 
a task far from being trivial in the frame proposed in \cite{T}. 
Indeed, one easily improves the order of accuracy of \ref{schem_ham} by considering a higher order approximation of
 the Hamiltonian and a higher order difference operator.\\
As a consequence, one is left with the problem of finding
a time discretization that preserves the Hamiltonian structure. It
is easily seen that an explicit forward Euler time integration is
unstable whereas the backward implicit Euler time integration is entropy
stable. In the linearized case, the Crank Nicolson scheme preserves
exactly the Hamiltonian. It would be interesting to consider various
symplectic integration schemes in time and in particular, consider various 
symplectic splitting strategies as it is now classical for the nonlinear Schrodinger
equation seen as an Hamiltonian PDE.

Note that the hamiltonian difference scheme derived
here is based on centered difference and it is well known that for
hyperbolic conservation laws, this could be a source of numerical
instabilities or spurious oscillatory modes. Though, the scheme considered
here also provide a control on the gradient of the density and thus
on oscillatory modes in addition of being more stable.

\section{\label{sec4}Numerical Simulations}

\subsection{Entropy stability: original vs new formulation}

Before carrying out a numerical simulation of an experiment by Liu and Gollub \cite{LG}
with the full shallow water system (\ref{svlg}), we have considered
the more simple situation of a fluid over an horizontal plane without
friction at the bottom. The shallow water system reads 
\begin{equation}
{\displaystyle \partial_{t}h+\partial_{x}(hu)=0,\quad\partial_{t}(hu)+\partial_{x}(hu^{2}+g\frac{h^{2}}{2})=\kappa h\partial_{xxx}h,}\label{ktest}
\end{equation}
where $g=9.8m.s^{-2}$ and $\kappa=\sigma/\rho$. The fluid under
consideration in \cite{LG} is an aqueous solution of glycerin with
density $\rho=1.134\, g.cm^{-3}$ and capillarity $\sigma=67\, dyn.cm^{-1}$.
%The augmented form of (\ref{ktest}) reads 
%\begin{equation}
%\begin{array}{lll}
%{\displaystyle \partial_{t}h+\partial_{x}(hu)=0,}\\
%{\displaystyle \partial_{t}(hu)+\partial_{x}(hu^{2}+g\frac{h^{2}}{2})=\sqrt{\kappa}\partial_{x}(h^{3/2}\partial_{x}w),}\\
%{\displaystyle \partial_{t}(hw)+\partial_{x}(huw)=-\sqrt{\kappa}\partial_{x}(h^{3/2}\partial_{x}u).}
%\end{array}\label{ktest_nf}
%\end{equation}
We first tested the entropy stability of (second order accurate) difference
approximations for the shallow water equations (\ref{ktest}) and
for its extended counterpart. We work
on a finite interval of length $X=80cm$ with periodic boundary conditions.
At time $t=0$, the fluid velocity $u=0$ and the fluid height is
given by ${\displaystyle h|_{t=0}=h_{N}\left(1+0.3\exp(-2000(x-0.4)^{2})\right)}$
with $h_{N}=1mm$ (the characteristic fluid height in experiments by Liu and Gollub).
In order to capture correctly the capillary ripples, we have chosen
$\delta x=0.25mm$ and $\delta t=120\delta x^{2}$. 
In figure \ref{fig1},
we draw the profile of the surface of the fluid at time $T=1s$.
%: as expected, the Lax Friedrichs scheme slightly damps the capillary ripples
%in front of the shocks. 
\begin{figure}[h!]
\begin{center}
\includegraphics[width=14.5cm, height=4.5cm]{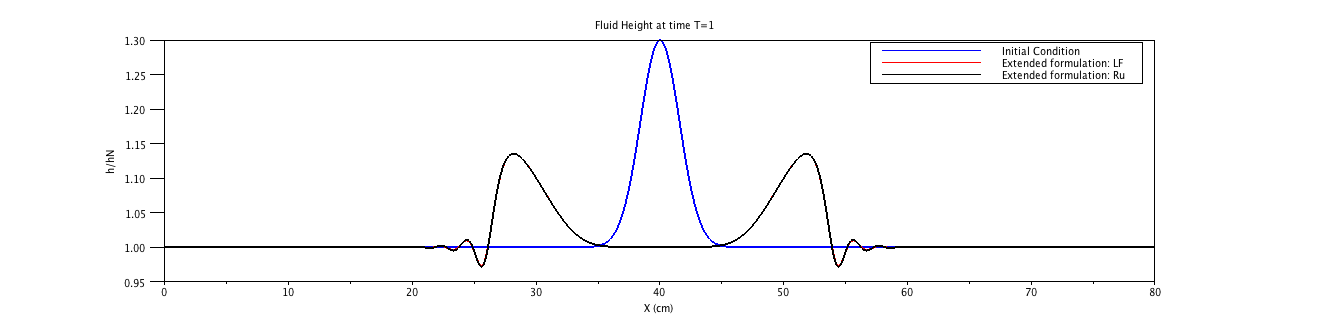} 
\par\end{center}
\caption{\label{fig1} Profile of the surface of the fluid at time $T=1$}
\end{figure}

\noindent In figure \ref{fig2}, we have drawn the relative entropy
$\frac{U}{U|_{t=0}}$ as a function of time: the picture clearly indicates
that the difference approximation of the extended formulation have better
entropy stability properties than difference approximation of (\ref{ktest}).

\begin{figure}[h!]
\begin{center}
\includegraphics[width=14cm,height=6cm]{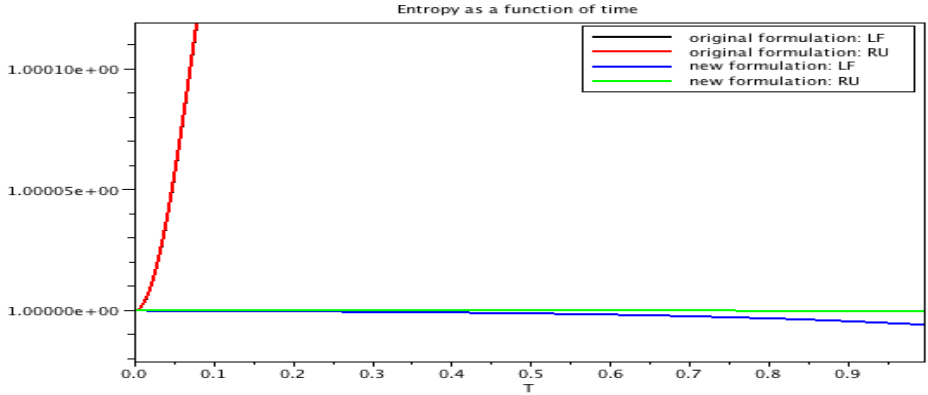} 
\par\end{center}

\caption{\label{fig2} Entropy as a function of time: Comparison of the various discretizations of (\ref{ktest}).
%As expected, the entropy is almost constant for the difference approximations
%of \ref{ktest_nf}. The results for the Harten-Lax-van Leer scheme
%are not plotted as the results are the same than the ones for the
%Rusanov scheme
}
\end{figure}

\noindent
A natural question arises about the new formulation: indeed one may
ask whether the relation $hw=\frac{2}{3}\sqrt{\kappa}\partial_{x}(h^{3/2})$
is satisfied for all time. If not, it does not make sense to compare
the performance with respect to entropy stability since it would represent
two distinct quantities. In figure \ref{fig3}, we draw the relative
error at time $T=1$ and defined as 
\[
{\displaystyle err_{j}=\frac{|(hw)_{j}-\sqrt{\kappa}\frac{h_{j+1}^{3/2}-h_{j-1}^{3/2}}{3\delta x}|}{\|hw\|},\quad j=1,\dots,N.}
\]
The numerical simulations show very good agreement, especially for
the less dissipative scheme, Rusanov, than
for Lax Friedrichs scheme. We have also implemented an alternative scheme where the relation 
$hw=\frac{2}{3}\sqrt{\kappa}\partial_{x}(h^{3/2})$ is enforced 
at each time step: we have not noticed any change in the numerical solution.

The CFL condition found for Rusanov scheme is of the form $\delta t=O(\delta x^{7/3})$
that is rather close to the ``optimal'' heuristic CFL condition
$\delta t=O(\delta x^{2})$. Therefore, we can conclude that a difference approximation of the extended formulation of (\ref{ktest})
with a Rusanov flux and second order accurate both in time and space
is a natural candidate to perform numerical simulations of falling films
experiments by Liu and Gollub \cite{LG}.

\noindent 
\begin{figure}[h!]
\begin{centering}
\includegraphics[width=14.5cm,height=4.5cm]{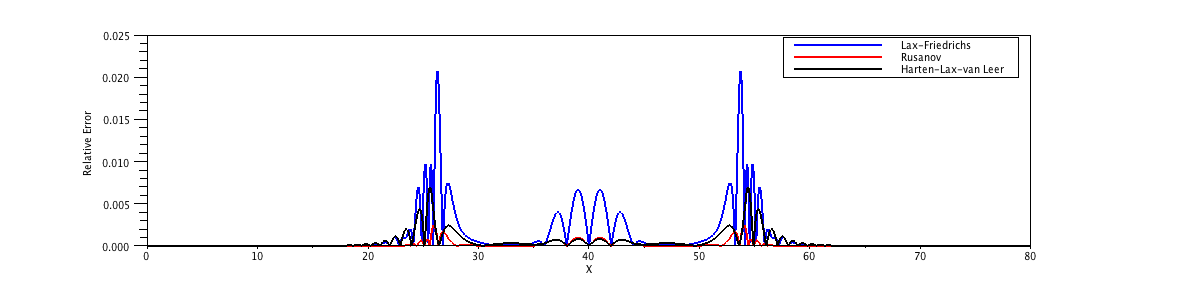} 
\par\end{centering}

\caption{\label{fig3} Consistency of the new formulation: relative error between
the new variable $hw$ and $\frac{2}{3}\sqrt{\kappa}\partial_{x}(h^{3/2})$.
The Rusanov scheme and the Harten Lax Van Leer scheme have comparable
consistency properties. In regards, the Lax Friedrichs scheme is less
efficient to preserve consistency}
\end{figure}

\subsection{Hamiltonian discretization and dispersive shock waves}

In what follows, we have tested the difference hamiltonian approximation
of (\ref{ktest}). The initial conditions are the same than in the previous section.
In order to be entropy stable it is necessary to
employ an implicit method: we have used here an implicit backward
Euler time discretization. Due to the nonlinearity of the problem,
the Crank Nicolson, second order accurate, time discretization does
not guarantee entropy stability. Therefore, we did not try to compare
with other schemes tested in the previous section. An important remark
is that now there is no numerical viscosity: a drawback is that the
dynamical behavior is completely changed as shown in figure \ref{fig5}
in comparison to what is found in the presence of numerical viscosity (figure \ref{fig1}).

\begin{figure}[h!]
%\begin{centering}
%\includegraphics[height=4.5cm, width=17cm]{ktest_fh_hauteur} 
%\par\end{centering}

%\caption{\label{fig4} Fluid height profile at time T=0.26: after the breaking of the wave, an oscillatory part appears, due to dispersive regularization (formation of a dispersive shock wave)}

\begin{center}
\includegraphics[height=4.5cm, width=14cm]{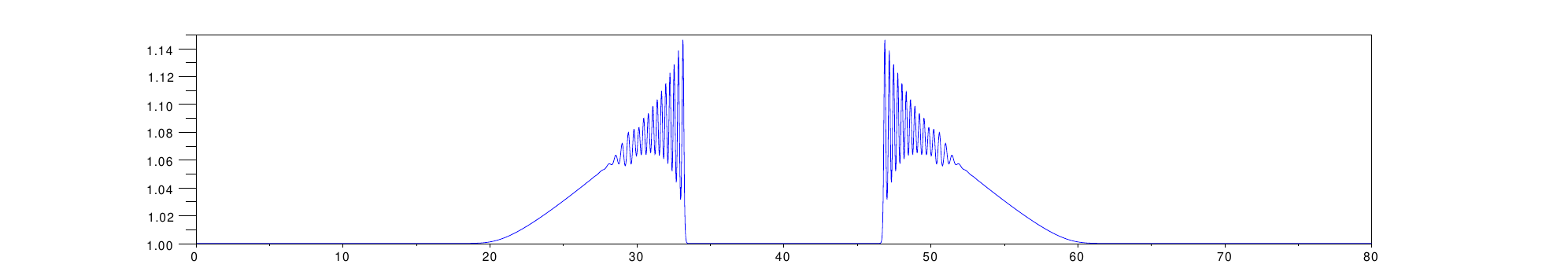} 
\par\end{center}

\caption{\label{fig5} 
%On top: Fluid height profile at time $T=0.26$: after
%the breaking of the wave, an oscillatory part appears, due to dispersive
%regularization. On below: 
Fluid height profile at time T=0.66: the oscillatory zone is increasing with time, characteristic of a dispersive shock wave}
\end{figure}

In order to see whether it is a numerical artifact, we checked the
entropy stability of the difference hamiltonian approximation: the
entropy remains clearly bounded with time. Indeed, these oscillations
are not a numerical artifact and can be explained (formally) by the
theory of dispersive shock waves. Here, the classical hyperbolic shocks
are smoothed by disperses effects: the oscillatory zone grows up in
time and the oscillations are described by the Whitham modulations
equations. This picture is not valid anymore in the presence of a
slight amount of viscosity: there are still some oscillations but
the width of the oscillatory zone stops growing after some time: see
\cite{J} and \cite{EGK} for a detailed analysis respectively in
the case of the Korteweg de Vries/Burgers equation and in the case
of the Kaup system perturbed by a viscous term. 
%In the previous section,
%a similar situation arises: the oscillatory zone stops growing as
%it is shown in the former computations in picture \ref{fig1}. 
Here the physical viscosity is replaced by numerical viscosity.

\begin{figure}[h!]
\begin{centering}
\includegraphics[width=13cm, height=5cm]{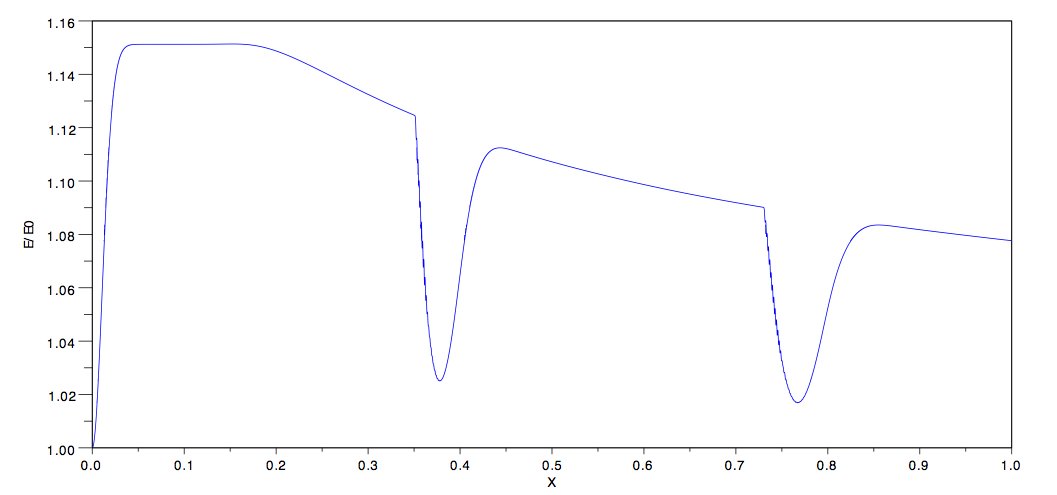} 
\par\end{centering}

\caption{\label{fig6} Entropy as a function of time. The behavior is rather
similar to difference approximations in original variables: the entropy
first increases then decreases with time: the ``holes'' in the decreasing
part of the curve correspond to times when the bumps interact}
\end{figure}

\subsection{Simulation of a Liu-Gollub experiment}

In this section, we show a numerical simulation for a shallow water
model derived for thin film flows down an inclined plane. The model is written as (see \cite{BN} for more details)
\begin{equation}\label{svlg}
\begin{array}{ll}
\displaystyle
\partial_t h+\partial_x(hu)=0,\\
\displaystyle
\partial_t(hu)+\partial_x\left(hu^2+P(h)\right)=\frac{2}{9\varepsilon Re}(h-\frac{u}{h})+\frac{\varepsilon}{We}h\partial_{xxx}h+\frac{6\varepsilon}{Re}\partial_{xx}(hu),
\end{array}
\end{equation}
with $P(h)$ is a pressure term given by
$
\displaystyle
P(h)=\frac{h^2}{2F^2}+\frac{2h^5}{25}.
$
The non dimensional $Re$, $F^2$, $We$ and $\varepsilon$ are, respectively the Reynolds, Froude and Weber numbers and the aspect ratio. Once $Re$ is fixed, we define
$h_N, u_N, F^2$ and $We$ as
$$
\displaystyle
h_N=\left(\frac{2Re \nu^2}{g\sin\theta}\right)^{\displaystyle\frac{1}{3}}, u_N=\nu\frac{Re}{h_N}, \quad F^2=\frac{2}{9}Re\tan(\theta), \quad We=\frac{\rho\lambda u_N^2}{\sigma} 
$$
with $\lambda$ a characteristic wavelength of the flow, $\sigma$ represents the surface tension of the fluid, $\nu$ its viscosity and  $\rho$ its density. 
Since the capillary ripples found in \cite{LG} are of order $1cm$, we choose $\lambda=0.01$.
The characteristic time scale is $T_N=\lambda/u_N$. In \cite{LG}, the Reynolds number in the experiment is $Re=29$ whereas $g=9.8$, $\rho=1134$, 
$\theta=6.4^{o}$, $\sigma=6.7\times 10^{-2}$ and $\nu=6.28\times 10^{-6}$. Then, one finds
$$
\displaystyle
h_N\approx 1.28\times 10^{-3},\:\: u_N\approx 9.49\times 10^{-2},\:\: T_N\approx 0.105,\:\: F^2\approx 0.723,\:\: We\approx 1.52.
$$
%\begin{equation}\label{svlg}
%\begin{array}{ll}
%{\displaystyle \partial_{t}h+\partial_{x}(hu)=0,}\\
%{\displaystyle \partial_{t}(hu)+\partial_{x}\left(hu^{2}+P(h,A_1)\right)=A_1\left(gh\sin(\theta)-3\frac{\nu u}{h}+\frac{\sigma}{\rho}h\partial_{xxx}h\right)+4\nu\partial_{xx}q,}
%\end{array}
%\end{equation}
%\noindent
%where $A_1>0$ and $P(h,A_1)$ is a pressure term given by 
%$$
%\displaystyle
%P(h,A_1)=A_1g\cos(\theta)\frac{h^2}{2}+(\frac{4}{45}-\frac{2A_1}{25})\left(\frac{g\sin\theta}{\nu}\right)h^5.
%$$
%\noindent
Note that the viscous term $(6\varepsilon/Re)\partial_{xx}(hu)$ is only heuristic so that the model is not a second order accurate model (with respect to the aspect ratio $\varepsilon=h_N/\lambda$).The frequency of the perturbation at the inlet is $f=1.5Hz$. At time $t=0$, the fluid height and velocity are constant $h=1$ and $u=1$. Following the conclusions of our study, we have chosen to carry out numerical simulations with a fully second order accurate scheme of the extended formulation of (\ref{svlg}) and used a Rusanov flux for the first order part. Our numerical results show a good agreement with the experiment by Liu and Gollub \cite{LG}.

\noindent Up to now, the choice of boundary conditions for the Euler Korteweg
equations on a finite interval is an open problem so that we have
chosen rather arbitrary boundary conditions. Furthermore, since the
difference scheme contains numerical/physical viscosity, we have considered a set of $5$
boundary conditions. First, at the inlet, we chose: 
$
{\displaystyle h|_{x=0}=1+0.03\sin(2\pi\, f\,T_N t),\quad hu|_{x=0}=1,\quad\partial_{x}h|_{x=0}=0.}
$
In contrast to \cite{KRSV}, we have chosen free boundary conditions
at the outlet: 
\[
{\displaystyle \partial_{x}h|_{x=L}=\partial_{x}(hu)|_{x=L}=0}
\]
instead of ``hyperbolic type'' boundary conditions where $h$ and
$hu$ are convected with an artificial velocity $V_{out}>0$. As pointed
out in \cite{KRSV} the choice of the boundary conditions at the outlet
does not seem to influence the dynamic within the channel (no reflection
waves).

\noindent 
\begin{figure}[h!]
%\begin{centering}

\includegraphics[height=4cm, width=14cm]{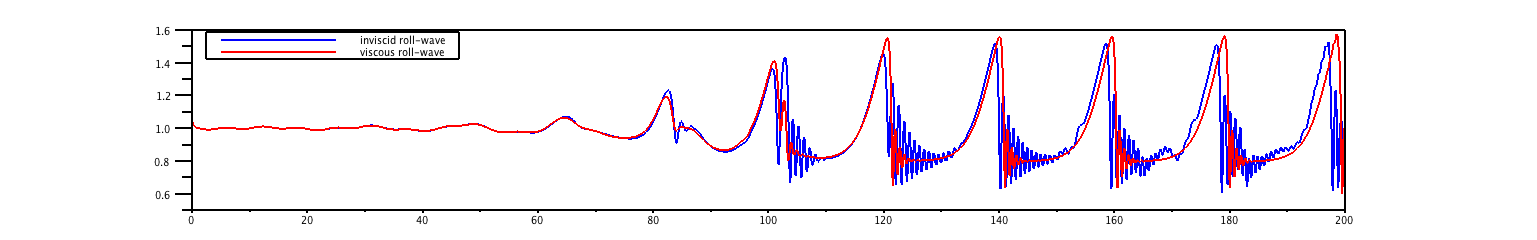} 
%\par\end{centering}

%\begin{centering}
\includegraphics[height=4cm, width=14cm]{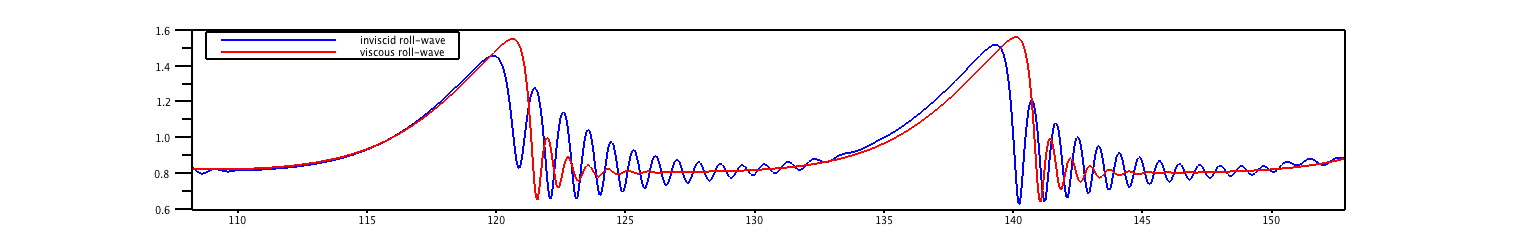} 
%\par\end{centering}

\caption{\label{fig7} Simulation of Liu Gollub experiment \cite{LG}: the
Reynolds number is $Re=29$ and the inclination is $\theta=6.4^{o}$.
The frequency at the inlet if $f=1.5Hz$. On top: a picture of the
complete experiment, from the inlet to the outlet (2m). On below:
a zoom over one spatial period when roll-wave profiles are stabilized}
\end{figure}

\section{Concluding remarks}

In this paper, we considered the stability of various difference approximations
of the Euler Korteweg equations with applications to shallow water
equations with surface tension. A first class of difference approximations
is built by considering the Euler Korteweg system as the classic isentropic
compressible Euler equations perturbed by a disperse term. This latter
term is discretized with centered finite differences and various classical
scheme for the convection part are considered. It is proved that a
certain amount of numerical viscosity is needed to obtain difference
schemes that are stable in the von Neumann sense (under suitable CFL
conditions).

In order to get entropy stability, we considered an extended formulation
of the Euler Korteweg equations and proved entropy stability of Lax Friedrichs
type schemes whereas Roe/Godunov schemes are always unstable with forward Euler explicit 
discretization. We have shown numerically that the extended formulation
of the Euler Korteweg system has better stability properties than
the original one. We also carry out a numerical simulation of a shallow
water system which models an experiment by Liu and Gollub to observe
roll-waves \cite{LG}.

By considering the Euler Korteweg system as a Hamiltonian system of
evolution PDEs, we introduced a semi-discretized difference approximations
which preserves the Hamiltonian structure. This scheme has no numerical
viscosity so that it is particularly useful to study purely dispersive
Euler Korteweg system: in particular, one can find numerically the
dispersive shock waves \cite{E} of the Euler Korteweg system. 

Several questions remain open. First, we carried out a numerical simulation
of an experiment of Liu and Gollub \cite{LG} by choosing arbitrary boundary
conditions. In fact, the choice of suitable boundary conditions for
the Euler Korteweg system on a finite interval in order to prove well
posedness is still an open problem. A first attempt in this direction
is found in \cite{A} where the well posedness of the linearized Euler Korteweg
equations is proved on a half space under a generalized Lopatinskii
condition.

Furthermore, we restricted our attention to one dimensional problem.
For thin film flows, this restricts the study to primary instabilities:
in order to analyze secondary instabilities found, one has consider $2d$ problems.
 In that setting, an extended formulation
is still available \cite{BDDd} so that we expect our analysis extends
easily, at least to cartesian meshes. An other interesting question
is the extension of this analysis to other mixed hyperbolic/dispersive
equations like the Boussinesq equations or the Serre/Green-Naghdi
equations. Up to now, the strategy adopted to deal with these system
is time splitting without proof of stability (though numerical results
are rather satisfying).

Finally an other open interesting question concerns the time integration
of the hamiltonian semi-discrete approximation: here, we have used
a backward Euler time integration so as to be entropy stable
but it does not preserve the hamiltonian (nor a perturbation of
it). Instead, one should consider symplectic time integration scheme,
 in particular but using various splitting. This kind of method are particularly
of interest in order to study the nonlinear stability of various traveling 
waves solutions of the purely dispersive equations.


\begin{thebibliography}{BDDd}

\bibitem[A]{A} C. Audiard \textit{Non homogeneous boundary value
problems for linear dispersive equations}, Comm. Parti. Diff. Eqs
37 (2012) no 1, p. 1-37

\bibitem[BDD]{BDD} S. Benzoni-Gavage, R. Danchin, S. Descombes \textit{Well-posedness
of one-dimensional Korteweg models}, Electronic J. Diff. Equations
59 (2006), pp. 1-35.

\bibitem[BDDd]{BDDd} S. Benzoni-Gavage, R. Danchin, S. Descombes
\textit{On the well-posedness for the Euler-Korteweg model in several
space dimensions}, Indiana Univ Math Journal 56 (2007), no 4, pp. 1499-1579.

\bibitem[BDK]{BDL} D. Bresch, B. Desjardins, C.K. Lin
\textit{On some compressible fluid models: Korteweg, Lubrication and Shallow Water systems}
Comm. Part. Diff Eqs 28 (2003) no. 3-4, p. 843-868.

\bibitem[BN]{BN} D. Bresch, P. Noble
\textit{Mathematical Justification of a shallow water model}, Methods and Applications of Analysis 14 (2007) no. 2, p. 87-117.

\bibitem[CDS]{CDS} R. Carles, R. Danchin, J.-C. Saut
\textit{Madelung, Gross-Pitaevskii and Korteweg} (2011) arXiv:1111.4670.

\bibitem[CL]{CL} C. Chalons, P.G. LeFloch
\textit{High-Order Entropy-Conservative Scheme and Kinetic Relations for Van der Waals Fluids}, Journal of Computational Physics 168 (2001) 184-206.

\bibitem[Col]{CoL} F. Coquel, P.G. LeFloch
\textit{An entropy satisfying MUSCL scheme for systems of conservation laws} Numer. Math. 74 (1996) 1-33. 

\bibitem[E]{E} G. A. El, \textit{Resolution of a shock in hyperbolic
systems modified by weak dispersion} Chaos, 15 (3): 037103, 21, 2005.

\bibitem[EGK]{EGK} G. A. El, R. H. J. Grimshaw, and A.M. Kamchatnov,
\textit{Analytic model for a weakly dissipative shallow water undular bore}
Chaos, 15 (3): 037102, 2005.

\bibitem[EGS]{EGS} G. A. El, R. H. J. Grimshaw, and N. F. Smyth.
\textit{Unsteady undular bores in fully nonlinear shallow-water theory.
} Physics of Fluids, 18(2):027104, February 2006.

\bibitem[HR]{HR} J. Haink, C. Rohde
\textit{Local Discontinuous-Galerkin Schemes for Model Problems in Phase Transition Theory}
Commun. Comput. Phys, 4  (2008) no 4, p. 860-893.

\bibitem[HAC]{HAC} M.A. Hoefer, M.J. Ablowitz, I. Coddington, E.A. Cornell, P. Engels, V. Schweikhard
\textit{Dispersive and classical shock waves in Bose-Einstein condensates and gas dynamics}
Physical Review A 74 (2006) no 2, 023623.

\bibitem[JTB]{JTB} D. Jamet, D. Torres and J.U. Brackbill
\textit{On the theory and Computation of Surface Tension: The elimination of Parasitic Currents through Energy Conservation in the Second-Gradient Method}
Journal of Computational Physics 182 (2002) 262-276.


\bibitem[J]{J} S. Johnson \textit{A non-linear equation incorporating
damping and dispersion} J. Fluid. Mech. 42 (1970) p. 49-60.

\bibitem[KRSV]{KRSV} S. Kalliadasis, C. Ruyer-Quil, B. Scheid, M.G.
Velarde \textit{Falling Liquid Films}, Applied Mathematical Sciences
176 (2012)

\bibitem[LMR]{LMR} P.G. LeFloch, J.M. Mercier, C. Rohde \textit{Fully
discrete, entropy conservative schemes of arbitrary order}, SIAM J.
Numer. Anal. 40 (2002) No 5, pp 1968-1992.

\bibitem[LG]{LG} J. Liu, J.P. Gollub \textit{Solitary wave dynamics
of film flows}, Phys. Fluids 6, 1702 (1994).

%\bibitem[LSG]{LSG} J. Liu, B. Schneider, J.P. Gollub \textit{Three-dimensional
%instabilities of film flows}, Phys. Fluids 7, 55 (1995).

\bibitem[NV]{NV} P. Noble, J.-P. Vila \textit{Consistency, Stability,
Gallilean indifference criteria for design of Two moments closure
equations of Shallow water type for Thin Film laminar flow gravity
driven.} in preparation.

\bibitem[VL]{VL} B. Van Leer
\textit{Towards the ultimate conservative difference schemes: a second order sequel to Godunov's method} J. Comp. Phys 43 (1981) 357-372.

\bibitem[T]{T} E. Tadmor \textit{Entropy stability theory for difference
approximations of nonlinear conservation laws and related time-dependent
problems} Acta Numerica (2003) pp. 451-512.

\bibitem[YS]{YS} J. Yan, C.-W. Shu
\textit{A local discontinuous galerkin method for KdV type equations}
SIAM J. Numer Anal 40 (2002) no. 2, p. 769-791.

\end{thebibliography}
\end{document}